\documentclass[reqno]{amsart}
\newcommand \datum {July 10, 2022}

\usepackage{amssymb,latexsym}
\usepackage{amsmath}
\usepackage{graphicx}
\usepackage[dvipsnames]{xcolor}
\usepackage{enumerate}
\numberwithin{equation}{section}

\theoremstyle{plain}
 \newtheorem*{namedtheorem}{\theoremname}
 \newcommand{\theoremname}{testing}
 
 \newtheorem{theorem}{Theorem}[section]
  
 \newtheorem{lemma}[theorem]{Lemma}
 
 \newtheorem{observation}[theorem]{Observation} 
 \newtheorem{corollary}[theorem]{Corollary}
  
\theoremstyle{definition}
 \newtheorem{definition}[theorem]{Definition}

 \newtheorem{remark}[theorem]{Remark}

\theoremstyle{remark}

 \newtheorem{algorithm}[theorem]{Algorithm}

\newenvironment{enumeratei}{\begin{enumerate}[\upshape (i)]}%
                            {\end{enumerate}}

\newcommand \rdelta {\mathrel{\delta}}
\newcommand \rlambda {\mathrel{\lambda}}
\newcommand \rblambda {\mathrel{\lambda_{\textup{bnd}}}}
\newcommand \Jir [1] {\textup{J}(#1)} 
\newcommand \Mir [1] {\textup{M}(#1)}
\newcommand \CTF [1] {\textup{CTF}(#1)}
\newcommand \CDE [1] {\textup{CDE}(#1)}

\renewcommand \phi{\varphi}
\newcommand \restrict [2] {{#1}\kern-1pt \rceil_{\kern-1pt #2}}
\newcommand \ZZ [1] {\mathbb Z_{#1}}
\newcommand \ZK [1] {\frac 12\mathbb Z_{#1}}

\newcommand \gstparallel {\mathrel{\parallel_{\textup {g}\lambda\rho}}}
\newcommand \gideal {\mathord{\downarrow_{\textup{g}}}}

\newcommand \brho {{{\rho}}}
\newcommand \bnu {{{\nu}}}
\newcommand \ljc[1] {\textup{ljc}(#1)}
\newcommand \rjc[1] {\textup{rjc}(#1)}
\newcommand \Lbut [1] {L^{-#1}}
\newcommand \CircR [1] {\textup{CircR}(#1)}
\newcommand \LEA [1] {\textup{LEA}(#1)}
\newcommand \REA [1] {\textup{REA}(#1)}

\newcommand \ntn {\intv n}
\newcommand \ntm {\intv m}
\newcommand \GInt [1] {\textup{GInt}(#1)}
\newcommand \geomeq {\mathrel{\overset{\textup{geo}}=}}
\newcommand \length {\textup{length}}
\newcommand \height {\textup{height}}
 
\newcommand \leftb [1]  {\textup{C}_{\textup{left}}(#1)} 
\newcommand \rightb [1] {\textup{C}_{\textup{right}}(#1)} 
\newcommand \ideal[1]{\mathord\downarrow\kern0.5pt #1}

\newcommand \bdia {\mathcal C_1}
\newcommand \tbdia {$\mathcal C_1$}

\newcommand \cdia {\mathcal C_2}
 
\newcommand \intv [1]{{\mathfrak #1}}  
\newcommand \Foot [1] {\textup{Foot}(#1)}
\newcommand \tuple [1] {(#1)}
\newcommand \pair [2] {\tuple{#1,#2}}

\newcommand \Enl [1] {\textup{Lit}(#1)}
\newcommand \LEnl [1] {\textup{LeftLit}(#1)}
\newcommand \REnl [1] {\textup{RightLit}(#1)}

\newcommand \Roof [1] {\textup{Roof}(#1)}
\newcommand \Floor [1] {\textup{Floor}(#1)}

\newcommand \LFloor [1] {\textup{LF}(#1)}

\newcommand \RFloor [1] {\textup{RF}(#1)}
\newcommand \NTube[1] {\textup{NTube}(#1)}
\newcommand \Lamp[1] {\textup{Lamp}(#1)}

\newcommand \rhbody {\brho_{\textup{Body}}}

\newcommand \rhgeomc {\brho_{\textup{CircR}}}

\newcommand \rhfoot {\brho_{\textup{foot}}}
\newcommand \rhinfoot {\brho_{\textup{infoot}}}

\newcommand \rhalg {\brho_{\textup{alg}}}
\newcommand \EnS [1] {\textup{LitSet}(#1)}
\newcommand \rbl {\mathrel{\beta_{\textup{left}}}}
\newcommand \rbr {\mathrel{\beta_{\textup{right}}}}
\newcommand \rbm {\mathrel{\beta_{\textup{mid}}}}

\newcommand \nucircr {\bnu_{\textup{LRCircR}}}

\newcommand \nuinfoot {\bnu_{\textup{infoot}}}

\newcommand \cirrec [1] {\textup{CircR}(#1)}
\newcommand \Body [1] {\textup{Body}(#1)}
\newcommand \Max [1] {\textup{Max}(#1)}
\newcommand \Peak [1] {\textup{Peak}(#1)}
\newcommand \cov [1]  {#1^+}
\newcommand \defiff {\overset{\textup{def}}{\iff}}

\DeclareMathOperator{\Con}{Con}
\newcommand{\tbf}{\textbf}
\newcommand{\set}[1]{\{#1\}}
\newcommand{\con}[1]{\textup{con}(#1)}

\newcommand \ujs[1] {}
\newcommand \cskjsg [1] {}

\newcommand\lred[1]{{\textcolor{red}{#1}\color{black}}}
\newcommand \red [1] {{\color{red!75!black}#1\color{black}}}

\newcommand \nothing [1] {}

\newcommand \hurem[1]{}

\begin{document}

\title[Congruences and lamps of slim semimodular lattices]
{Notes on congruence lattices and lamps of slim semimodular lattices}
\author[G.\ Cz\'edli]{G\'abor Cz\'edli}
\email{czedli@math.u-szeged.hu}
\urladdr{http://www.math.u-szeged.hu/~czedli/}
\address{University of Szeged, Bolyai Institute. 
Szeged, Aradi v\'ertan\'uk tere 1, HUNGARY 6720}

\begin{abstract} Since their introduction by G.\ Gr\"atzer and E.\ Knapp in 2007, more than four dozen papers have been devoted to finite slim planar semimodular lattices (in short, SPS lattices or slim semimodular lattices) and to some related fields. In addition to distributivity, there have been seven known properties of the congruence lattices of these lattices. 
The first two properties were proved by G.\ Gr\"atzer, the next four by the present author, while the seventh was proved jointly by G.\ Gr\"atzer and the present author.
Five out of the seven properties were found and proved by using lamps, which are lattice theoretic tools introduced by the present author in a 2021 paper. 
Here, using lamps, we present infinitely many new properties. 
Lamps also allow us to strengthen the seventh previously known property, and they lead to an algorithm of exponential time to decide whether a finite distributive lattice can be represented as the congruence lattice of an SPS lattice. Some new properties of lamps are also given.
\end{abstract}

\thanks{This research  was supported by the National Research, Development and Innovation Fund of Hungary, under funding scheme K 134851.}

\subjclass {06C10\hfill{\lred{\tbf{\datum}}}}

\dedicatory{Dedicated to Professor \'Agnes Szendrei on the occasion that she has recently become 
a member of the Hungarian Academy of Sciences}


\maketitle    

\section{Introduction}\label{S:Introduction}
\subsection{Outline and targeted readership}
The reader is assumed to be familiar with the rudiments of lattice theory. Two open access papers,  Cz\'edli \cite{CzGlamps} and Cz\'edli and Gr\"atzer \cite{CzGGG3p3c}, will be frequently referenced; they should be at hand.  When the terminology in these two papers are different, we give preference to \cite{CzGlamps}. 

The paper is structured as follows. In Subsection \ref{subsect:zsjWh} of the present section, we give a short survey to explain our motivations. 
 
In Section \ref{section-lamps-tubes}, we give an upper bound on the number of neon tubes (equivalently, on the number of trajectories) that are sufficient to represent a finite distributive lattice $D$ as the congruence lattice $\Con L$ of an SPS lattice $L$. This yields an upper 
bound on the smallest $|L|$ such that $L$ is an SPS lattice with $D\cong\Con L$ and offers an algorithm of exponential time to decide if there exists such an $L$. 

Section \ref{sect:notes-alg} comments the algorithm, which is easy to understand but it seems to be too slow for any practical purpose.

Section \ref{sect:illum-alg} outlines a complicated algorithm based on lamps (on sets illuminated by lamps to be more precise); note that not every detail of this algorithm is elaborated. 

Section \ref{sect:lampslemmas} proves some easy lemmas about lamps. 

Section \ref{sect:CTF} gives an infinite family 
of new properties, and proves that the congruence lattices of slim semimodular lattices have these properties; see Theorem \ref{thm:fncZ}, one of the main results. 

Section \ref{sect:CMP} gives another infinite family of properties and proves Theorem \ref{thm:CMP}, the second main result, which asserts that the congruence lattices of slim semimodular lattices have these properties.

Finally, the new and old properties are compared in  Section \ref{sect:conclR}.

\red{Note  that  few changes were only necessary to obtain the present version from the earlier version of June 29, 2022. Apart possibly from some insignificant ones, the changes are in red (this colour) and they are included (or referenced) in Definition \ref{def:cmpn} and in  Section \ref{sect:conclR}. }

\subsection{A  short survey and our goal}\label{subsect:zsjWh}
The introduction of slim planar semimodular lattices, \emph{SPS lattices} or \emph{slim semimodular lattices} in short, by Gr\"atzer and Knapp \cite{GKn07} in 2007 was a milestone in the theory of (planar) semimodular lattices. Indeed, \cite{GKn07} was followed by more than four dozen papers in one and a half decades; see

\centerline{\texttt{http://www.math.u-szeged.hu/\textasciitilde{}czedli/m/listak/publ-psml.pdf}} 

\noindent for the list of these papers. For motivations to study these lattices and also for their impact on other parts of mathematics, see the survey section, Section 2, of the open access paper Cz\'edli and Kurusa \cite{CzGKA19}.

By Gr\"atzer and Knapp's definition given in \cite{GKn07}, an \emph{SPS lattice} is a finite planar semimodular lattice that has no sublattice (equivalently, no cover-preserving sublattice) isomorphic to $M_3$. Later, by Cz\'edli and Schmidt \cite{CzGSchT11},
\emph{slim} lattices were defined as finite lattices $L$ such that  the poset (= partially ordered set) $\Jir L$ of the join-irreducible elements of $L$ is the union of two chains. These lattices are necessarily planar. 
It appeared in \cite{CzGSchT11} that SPS lattices are the same as \emph{slim semimodular lattices}.

In 2016, Gr\"atzer \cite{gGCFL2} and \cite{gG16} raised the problem \emph{what the congruence lattices of SPS lattices are}. These congruence lattices are finite distributive lattices, of course, but in spite of seven of their additional properties discovered so far in 
Gr\"atzer \cite{gG16} and \cite{gG20}, Cz\'edli \cite{CzGlamps}, and Cz\'edli and Gr\"atzer \cite{CzGGG3p3c}, we still cannot characterize them  in the language of lattice theory. Neither can we do so within the class of finite lattices; however, all the known properties can be given by a single axiom described in  Cz\'edli ~\cite{CzGfinax}.

\section{Lamps and the Neon Tube Lemma}\label{section-lamps-tubes}

\subsection{\tbdia-diagrams and slim rectangular lattices}
Let us recall some notations and concepts.
For an element $u\neq 1$ of a finite lattice $L$, let $\cov u$ denote the join of all covers of $u$, that is,
\begin{equation}
\cov u:=\bigvee\set{y\in L: u\prec y}.
\label{eqyyZswrRsHnRbp}
\end{equation}
Of course, if $u$ belongs to $\Mir L$, the set of (non-unit)  meet-irreducible elements, then exactly one joinand occurs in \eqref{eqyyZswrRsHnRbp}. 
A \emph{\tbdia-diagram} is a planar lattice diagram in which 
\begin{itemize}
\item for each $u\in \Mir L$ such that $u$ is in the (geometric, that is, topological) interior of the diagram, $[u,\cov u]$ is a \emph{precipitous edge}, that is, the angle measured from the (positive half) of the $x$ coordinate axis to the edge $[u,\cov u]$ is strictly between $\pi/4$ ($45^\circ$) and $3\pi/4$ ($135^\circ$),
\item and any other edge is of \emph{normal slope}, that is, the angle between the $x$-axis and the edge is $\pi/4$ or $3\pi/4$. 
\end{itemize}
All \emph{lattice} diagrams in the paper are \tbdia-diagram. (Poset diagrams also occur, for which ``\tbdia-diagrams''  are not even defined.) 
We know from Cz\'edli \cite{CzG132} that each slim semimodular lattice has a \tbdia-diagram. 
A slim semimodular lattice $L$ is \emph{rectangular} if it has exactly two \emph{corners}, that is, elements of $\Jir L\cap \Mir L$, and they are complementary; see Gr\"atzer and Knapp \cite{GKn09} for the introduction of this concept or see Cz\'edli \cite[page 384]{CzGlamps} where this concept is recalled.
If $L$ is so, then any of its \tbdia-diagrams is of a rectangular shape. Furthermore, each side of the \emph{full geometric rectangle} that the contour of the diagram determines is of a normal slope. 
All lattice diagrams in the paper are \tbdia-diagrams of slim rectangular lattices. 
In the rest of the paper, 
\begin{equation}\text{$L$ will always denote a slim rectangular lattice with a fixed \tbdia-diagram.}
\label{eq:cnFnwlHx}
\end{equation}
This assumption is justified by a result Gr\"atzer and Knapp \cite{GKn09}, which implies that
\begin{equation}\left.
\parbox{7.5cm}{for each SPS lattice $K$ there is slim rectangular lattice $L$ such that $\Con K=\Con L$;};\,\,\right\}
\label{eq:szhGkTnBldR}
\end{equation}
see also Cz\'edli \cite{CzG132}. Note that, to improve the outlook,  
$L_2$ in Figure \ref{fig1}, $K'$ in Figure \ref{fig4} and the lattices in Figures \ref{fig2} and \ref{fig5}  are given by  $\cdia$-diagrams. (A \tbdia-diagram is a $\cdia$-diagram if any two edges on the lower boundary are of the same geometric length.)

\begin{figure}[ht] \centerline{ \includegraphics[scale=1.0]{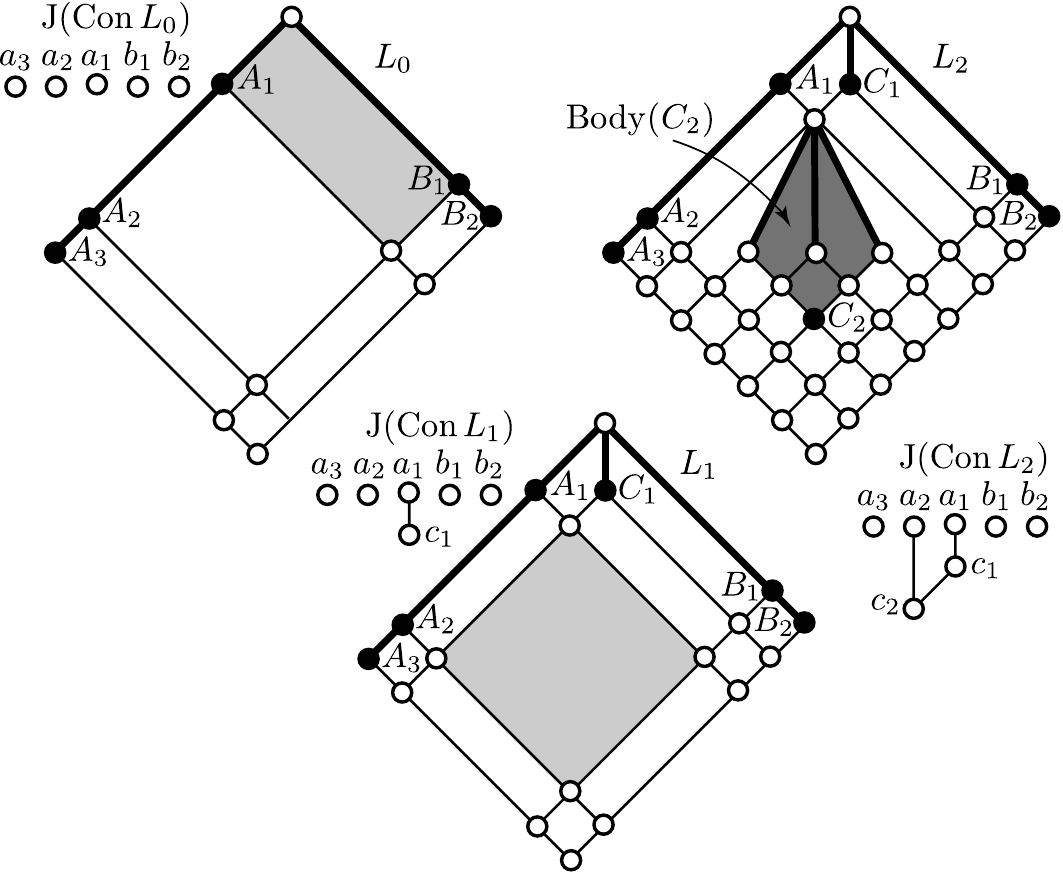}} \caption{A multifork extension}\label{fig1}
\end{figure} 

\subsection{Multiforks, lamps and related geometric objects}
A 4-\emph{cell} $X$, that is a cover-preserving 4-element boolean sublattice, is \emph{distributive} is so is the principal ideal $\ideal {1_X}:=\set{u\in L: u\leq 1_X}$. Given a 4-cell $X$ of $L$ and a positive integer $k$, we can insert a \emph{$k$-fold multifork} or, if $k$ is unspecified, a \emph{multifork} into $X$ to obtain a larger slim rectangular lattice,
which is called a \emph{multifork extension} of $L$;
see Cz\'edli \cite{CzG:pExttCol} where this concept is introduced, or see
(2.9) and Lemma 2.12 of Cz\'edli \cite{CzGlamps} where it is recalled, or see only Figure \ref{fig1} here. In this figure,
we add a 1-fold multifork (also called a \emph{multifork}) to the grey-filled 4-cell of $L_0$, and we obtain $L_1$. We obtain $L_2$ by adding a 3-fold multifork to the grey-filled 4-cell of $L_1$.

Next,  based on Cz\'edli \cite[Definitions 2.3 and 2.6--2.7]{CzGlamps}, we define neon tubes, lamps, and some related geometric  concepts.
By a \emph{neon tube} of $L$ we mean an edge $[u,\cov u]$ such that $u\in\Mir L$. 
The \emph{boundary neon tubes} are of normal slopes while the \emph{internal neon tubes} are precipitous. 
For a neon tube $\ntn=[u,\cov u]$, we denote $u$ and $\cov u$ by $\Foot {\ntn}$ and $\Peak{\ntn}$, respectively. 
Clearly, $\ntn$ is determined by its foot, $\Foot{\ntn}$.
A \emph{boundary lamp} $I$ is a single boundary 
neon tube $\ntn$. (However, we often say that the boundary lamp $I$ \emph{has} the neon tube $\ntn$.)  If $\ntn$ is an internal neon tube, then we let
\begin{equation}
\text{$\beta_{\ntn}:=\bigwedge\{\Foot{\intv m}: \intv m$ is an internal neon tube and $\Peak{\intv m}=\Peak{\ntn}\}$}
\label{eq:nwTwshmTflTg}
\end{equation}
and we say that $I:=[\beta_{\ntn}, \Peak{\ntn}]$ is an \emph{internal lamp} of $L$. The neon tubes $\intv m$ in \eqref{eq:nwTwshmTflTg} are the neon tubes of $I$. If $I$ and $\ntn$ are as above, we use the notations $\Foot I=\beta_{\ntn}$ and  $\Peak I=\Peak{\ntn}$.  By a \emph{lamp} (of $L$) we mean a boundary or internal lamp (of $L$). 
So lamps are particular intervals and each lamp is determined by its neon tubes. Actually, more is true since we know from Cz\'edli \cite[Lemma 3.1]{CzGlamps} that
\begin{equation}
\text{each lamp $I$ is determined by its foot, $\Foot I$.}
\end{equation}
This allows us to give the lamps of our diagrams by their feet; these feet are exactly the black-filled elements. See, for example, Figures \ref{fig1}, \ref{fig2}, \ref{fig3}, \ref{fig4}, and \ref{fig5}.
We put the name of a lamp close to its black-filled foot.

\begin{figure}[ht] \centerline { \includegraphics[scale=1.0]{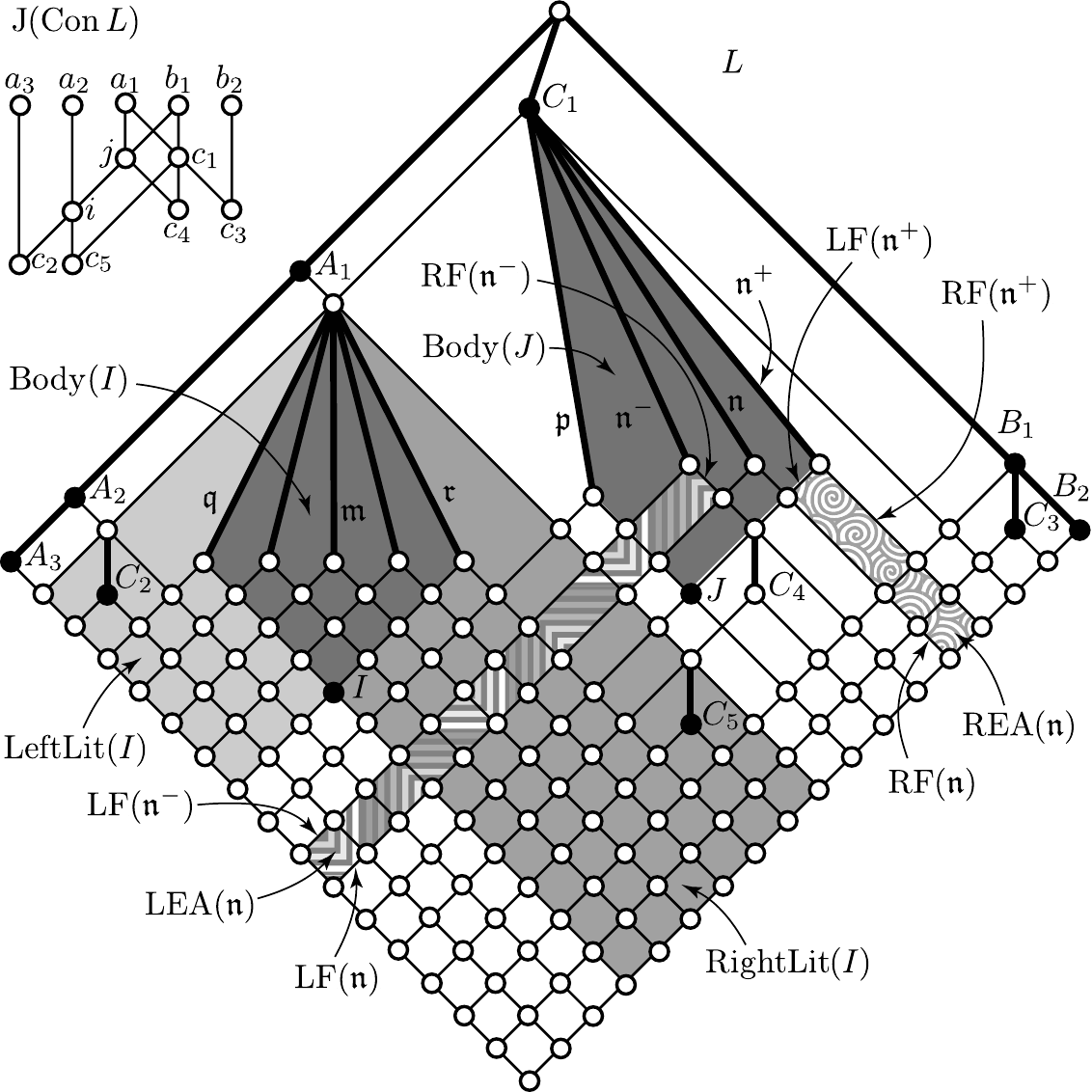}} \caption{Geometric objects related to lamps}\label{fig2}
\end{figure}

We know from Kelly and Rival \cite{KR75} that in a planar lattice diagram, each interval determines a geometric region.
As in  Cz\'edli \cite[Definition 2.6]{CzGlamps}, 
the \emph{body} of a lamp $I$, denoted by $\Body I$, is the geometric region determined by $I=[\Foot I, \Peak I]$. 
For example, for $L_2$ in Figure \ref{fig1} and its lamp $C_2$, $\Body {C_2}$ is filled by dark-grey. So are $\Body I$ and  $\Body J$ in Figure \ref{fig3}.
The region determined by the interval 
\[
\CircR I:=[\bigwedge_{x\prec\Foot I} x,\, \Peak I]
\]
is the \emph{circumscribed rectangle} of $I$; it is a rectangle with all the four sides of normal slopes, and it is always larger than $\Body I$. 
If $X$ is a neon tube or a lamp, then the  \emph{left floor} of $X$ is the closed line segments of (normal) slope $\pi/4$  between $\Foot X$ and the lower left boundary of the diagram; it is denoted by $\LFloor X$; see Figure \ref{fig2} for illustrations. The \emph{right floor} of $X$, denoted by $\RFloor X$, is  analogously defined. For a lamp $I$ and a neon tube $\ntn$, their \emph{floors} are defined by 
\begin{equation}
\Floor I:=\LFloor I\cup \RFloor I\quad\text{ and }\quad
\Floor \ntn:=\LFloor \ntn\cup \RFloor \ntn.
\label{eq:floorsxTd}
\end{equation}
The direct product of two finite non-singleton chains is called a \emph{grid}. 
We know from, say,  (2.9) and (2.10) of Cz\'edli \cite{CzGlamps} that for our slim rectangular lattice $L$,
\begin{equation}\left.
\parbox{11cm}{there exists a sequence $L_0,L_1,\dots,L_k=L$
of slim rectangular lattices and there are distributive 4-cells $H_i$
of $L_{i-1}$ such that $L_0$ is a grid and, for $1\leq i\leq k$, $L_i$ is obtained
from $L_{i-1}$ by inserting a multifork into $H_i$.
Furthermore, the internal lamps of $L$ originate from these multifork extensions and their circumscribed rectangles are $H_1,\dots,H_k$.
}\,\,\right\}
\label{eq:fztvkTz}
\end{equation}

\subsection{What are lamps good for?}
For the real answer to this subsection title, see 
part \eqref{lemma:NTLb} of Lemma \ref{lemma:NTL} later, which gives a tangible evidence of the importance of lamps.

In our model, each geometric point of a neon tube (as an edge) emits photons but these photon can only go downwards at degree $5\pi/4$ or $7\pi/4$. (That is, to southwest or southeast direction.)
For a neon tube $\ntn$, a geometric point $\pair x y$ of the full geometric rectangle of $L$ is \emph{illuminated by $\ntn$ from the left} if the neon tube (as a geometric line segment) 
has a nonempty intersection with the half-line $\set{\tuple{x-t, y+t}: 0\leq t\in \mathbb R}$. 
The set of geometric points of the full geometric rectangle that are illuminated by $\ntn$ from the left is denoted by $\REnl{\ntn}$.
(Note ``R'' in the acronym indicates the points illuminated from the left are on the right.) We define $\LEnl\ntn$ analogously.
For a lamp $I$ and a neon tube $\ntn$, we let
\begin{align}
&\REnl I:=\bigcup\set{\REnl \ntn: \ntn \text{ is a neon tube of }I},\label{eq:vmhlTkvta}\\
&\LEnl I:=\bigcup\set{\LEnl \ntn: \ntn \text{ is a neon tube of }I},\label{eq:vmhlTkvtb}\\
&\Enl{\ntn}:=\LEnl{\ntn}\cup\REnl{\ntn},\text{ and }
\Enl{I}:=\LEnl{I}\cup\REnl{I}.\label{eq:vmhlTkvtc}
\end{align}
In the acronyms above, ``Lit'' comes from ``light''. However,
in the text we prefer the verb ``illuminate'' because of its double meaning: our neon tubes emit physical \emph{light} and 
contribute a lot to our \emph{comprehension} of the congruence lattices of slim semimodular lattices.

\begin{figure}[ht] \centerline{ \includegraphics[scale=1.0]{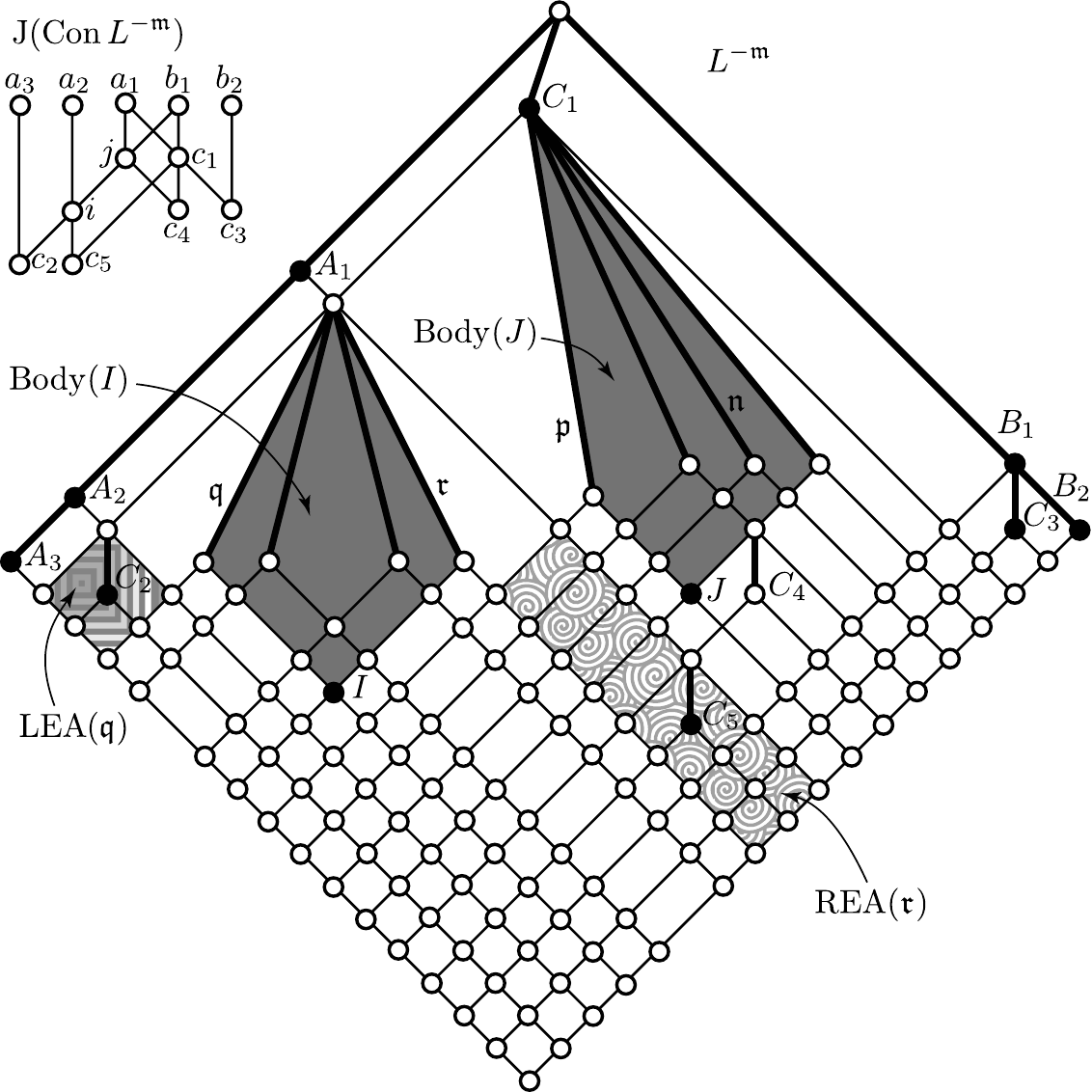}} \caption{$L^\ntm$,  obtained  from Figure \ref{fig2} according to \eqref{eq:wwHkhjkbb}}\label{fig3}
\end{figure}

While the geometric sets defined in \eqref{eq:vmhlTkvta}--\eqref{eq:vmhlTkvtc} here were sufficient for the proofs in Cz\'edli \cite{CzGlamps}, the present paper has to introduce some smaller sets as follows.
Let $\ntn$ be a neon tube. The unique lamp to which $\ntn$ belongs will be denoted by  $I=I_\ntn$. Below, we assume that $I$ is not a boundary lamp.
\begin{equation}\left.
\parbox{9.8cm}{If $\ntn$ is the leftmost neon tube of $I$, then $\LEA \ntn:=\LEnl \ntn\setminus \CircR I$. Otherwise, let $\ntn^-$ denote the left neighbour of $\ntn$ among the neon tubes of $I$, and 
let $\LEA \ntn$ be the (closed convex) geometric rectangle with sides $\LFloor \ntn$, $\LFloor {\ntn^-}$, a part of $\RFloor{\ntn^-}$, and a part of the lower boundary of $L$. We define $\REA\ntn$, the left-right symmetric counterpart of $\LEA \ntn$, analogously.}
\,\,\right\}
\label{eq:mrfZlbFsl}
\end{equation}
The choice of the  acronyms above will be explained a bit later.

For example, $\LEA\ntn$ and $\REA\ntn$ in  Figure \ref{fig2} are the zigzag-filled rectangle and the spiral-filled rectangle, respectively. In Figure \ref{fig3},
$\intv q$ is the leftmost neon tube of $I$ and 
$\LEA{\intv q}$ is the zigzag-filled rectangle while 
$\REA{\intv r}$ for the rightmost neon tube $\intv r$ is the spiral-filled rectangle. 
If $\ntn$ is a boundary lamp, then exactly one of $\LEA\ntn$ and $\REA\ntn$ is of positive geometric area while the other is a line segment or $\emptyset$.

For a subset $Y$ of the plane, let $\GInt Y$ denote the topological (in other words, geometric) interior of $Y$. Observe that
\begin{equation}\left.
\parbox{7.4cm}{for distinct neon tubes $\ntn$ and $\ntn'$ of a lamp, $\GInt{\LEA\ntn}\cap \GInt{\LEA{\ntn'}}=\emptyset$ and, analogously, 
$\GInt{\REA\ntn}\cap \GInt{\REA{\ntn'}}=\emptyset$}\,\,
\right\}
\label{eq:gvkpflLbvrlX}
\end{equation}
This motivates us to call $\LEA \ntn$ and $\REA \ntn$ the 
\emph{left exclusive area} and the \emph{right exclusive area} of $\ntn$; this is where the  acronyms in \eqref{eq:mrfZlbFsl} come from. Later, by an \emph{exclusive area} of $\ntn$ we mean one of   $\LEA \ntn$ and $\REA\ntn$.
By a trivial induction based on \eqref{eq:fztvkTz} or using \emph{trajectories} introduced in Cz\'edli and Schmidt \cite{CzGSchT11}, we obtain easily that
\begin{equation}
\parbox{8cm}{for any 4-cell $C$, there are unique neon tubes $\ntn$ and $\ntm$ such that $\LEA\ntn\cap\REA\ntm$ is the geometric rectangle determined by $C$.}
\label{eq:brzsfnnbrkzs}
\end{equation}

\subsection{On the number of neon tubes} In the whole paper, 
\begin{equation}\left.
\parbox{7.5cm}{$\Lamp L$ and $\NTube L$ denote the set of lamps and that of neon tubes of $L$, respectively.}\,\,\right\}
\label{eq:lMnTbjlm}
\end{equation}
On the set 
$\Lamp L$ of lamps of $L$, we define five relations; the first four are taken from  Cz\'edli \cite[Definition 2.9]{CzGlamps} while the last two are new. Note that using the remaining two out of the six relations of \cite[Definition 2.9]{CzGlamps}  together with  trivial inclusions like $\GInt X\subseteq X$ or, for a neon tube $\ntn$ of a lamp $I$, $\LEA\ntn\subseteq\LEnl I$, one could easily define even more relations (but this does not seem to be useful).

\begin{definition}[Relations defined for lamps]\label{defrhRhs} 
Let $L$ be a slim rectangular lattice with a fixed $\bdia$-diagram. For $I,J\in\Lamp L$, 
\begin{enumeratei}
\item\label{defrhRhsa} 
let $\pair I J\in\rhalg$ mean that  $\Peak I \leq \Peak J$, $I$ is an internal lamp, and $\Foot I\not\leq \Foot J$;
\item\label{defrhRhsb} 
let $\pair I J\in\rhfoot$  mean that  $I\neq J$, $\Foot I\in\Enl J$, and $I$ is an internal lamp;
\item\label{defrhRhs-bc} 
let $\pair I J\in\rhinfoot$  mean that  $I\neq J$, $\Foot I\in\GInt{\Enl J}$, and $I$ is an internal lamp;
\item\label{defrhRhsc} 
let $\pair I J\in\rhbody$  mean that  $I\neq J$, $\Body I\subseteq \Enl J$, and $I$ is an internal lamp;
\item\label{defrhRhsd} 
let $\pair I J\in\nuinfoot$  mean that  $I$ is an internal lamp, $I\neq J$ and $J$ has a neon tube $\ntn$ such that $\Foot I\in\GInt{\LEA\ntn}$
or $\Foot I\in\GInt{\REA\ntn}$; and, finally,  %

\item\label{defrhRhse} 
let $\pair I J\in\nucircr$  mean that  $I$ is an internal lamp, $I\neq J$ and $J$ has a neon tube $\ntn$ such that $\cirrec I\subseteq \LEA\ntn$ or $\cirrec I\subseteq \REA\ntn$.
\end{enumeratei}
\end{definition}

%

Now we are in the position to formulate the key lemma for this section. For $x,y\in L$, the least congruence containing $(x,y)$ is denoted by $\con{x,y}$.

\begin{lemma}[Neon Tube Lemma]\label{lemma:NTL}
 Let $L$ be a slim rectangular lattice with a fixed $\bdia$-diagram; then the following three assertions hold.
\begin{enumeratei}
\item\label{lemma:NTLa}  The six relations described in 
Definition \ref{defrhRhs} are all equal. Furthermore, they are equal to the relations given in Cz\'edli \cite[Definition 2.9]{CzGlamps}.
\item\label{lemma:NTLb}  Let $\leq$ denote the reflexive transitive closure of $\rhalg$. Then $\leq$ is a partial order and the poset $\tuple{\Lamp L;\leq}$ is isomorphic to the poset $\tuple{\Jir {\Con L};\leq}$ of nonzero join-irreducible congruences of $L$ with respect to the ordering inherited from $\Con L$. In fact,  the map
$\phi\colon \Lamp L\to \Jir{\Con L}$, defined by $[p,q]\mapsto \con{(p, q)}$, is an order isomorphism.
\item\label{lemma:NTLc} If $I\prec J$ (that is, $I$ is covered by $J$) in $\Lamp L$, then $\pair I J\in\rhalg$. 
\end{enumeratei}
\end{lemma}

Before the proof, several comments are reasonable. 
While $\rhfoot$ is the mildest geometric condition on $(I,J)$, $\nucircr$ is (seemingly) more restrictive that any other relation described in 
 \cite[Definition 2.9]{CzGlamps}. This is why the Neon Tube Lemma is a stronger than its counterpart, Lemma 2.11 of Cz\'edli \cite{CzGlamps}. 

In addition to lamps (and neon tubes), there are other approaches to the congruence lattices of slim rectangular lattices: the Swing Lemma from Gr\"atzer \cite{gG15} (see also Cz\'edli, Gr\"atzer and Lakser \cite{CzGGLakser} and Cz\'edli and Makay \cite{CzGMG17} for secondary approaches), the Trajectory Coloring Theorem from 
Cz\'edli \cite{CzG:pExttCol}, and even Lemma 2.36 (about the join dependency relation of Day \cite{day}, for any finite lattice) in Freese, Je\v zek and Nation \cite{fjnbook}.
Even though the differences among the four different approaches are not so big and most of these approaches would probably be  appropriate to prove the results of this paper on congruence lattices of slim semimodular lattices, we believe that our approach based on lamps (and neon tubes) gives the best insight  into the congruence lattices of slim rectangular (and, therefore, those of slim semimodular) lattices.  In addition to the present paper, this is witnessed by Cz\'edli \cite{CzGlamps} and Cz\'edli and Gr\"atzer \cite{CzGGG3p3c}. Indeed, with two early exceptions, all the known of these congruence lattices have been \emph{found} and first proved (or, at least, first proved) by lamps (and neon tubes).

\begin{proof}[Proof of Lemma \ref{lemma:NTL}] 
Recall that  Cz\'edli \cite[Definition 2.9]{CzGlamps} defines a relation $\rhgeomc$ on $\Lamp L$ as follows: a  $\pair I J\in\rhgeomc$ if $I$ is an internal lamp, $\cirrec I\subseteq \Enl J$, and $I\neq J$. 

It suffices to prove the first sentence of part 
 \eqref{lemma:NTLa} since the rest of the lemma follows from its counterpart, \cite[Lemma 2.11]{CzGlamps}, which also contains $\rhalg$, $\rhfoot$, $\rhinfoot$, and $\rhbody$.
Fortunately, the proof of \cite[Lemma 2.11]{CzGlamps} also proves the above-mentioned first sentence provided we observe the following. 

We know from \cite[Lemma 2.11]{CzGlamps} that $\rhalg=\rhfoot=\rhgeomc$. Assume that $(I,J)\in\rhgeomc$.
Using \eqref{eq:fztvkTz}, which is the combination of (2.9) and (2.10) of \cite{CzGlamps}, $J$ comes sooner than $I$. When $J$ has just arrived, the 
exclusive areas of its neon tubes are separated by edges. By (2.11) of \cite{CzGlamps}, these sets are still separated by edges when $I$ arrives. By planarity, these edges cannot cross $\cirrec I$.\footnote{In the proof of \cite[Lemma 2.11]{CzGlamps}, planarity was used in the same way; the only difference is that, apart from those 4-cells that are nondistributive since their tops is $\Peak J$, $\Enl J$ in \cite{CzGlamps} was only divided into two parts, $\LEnl J$ and $\REnl J$.} 
Hence the covering square into which $I$ enters (and which is geometrically $\cirrec I$) is a subset of an exclusive area of a neon tube of $J$. 

Keeping the above paragraph in mind, the (long) proof of  \cite[Lemma 2.11]{CzGlamps} works in the present situation. This completes the proof of the Neon Tube Lemma.
\end{proof}

Before formulating an easy consequence (under the name ``lemma'') of the Three Neon Tubes Lemma, we define two easy-to-understand concepts. A neon tube $\ntn$ of $L$ is \emph{secondary} if  there is no $I\in\Lamp L$ such that $\Foot I\in\GInt{\LEA\ntn}\cup\GInt{\REA\ntn}$. Equivalently, if 
for every $I\in\Lamp L$, neither $\cirrec I\subseteq\LEA\ntn$ nor $\cirrec I\subseteq\REA\ntn$. 
In the opposite case when there is an $I\in\Lamp L$ such that 
$\Foot I\in\GInt{\LEA\ntn}\cup\GInt{\REA\ntn}$,  we say that $\ntn$ is a \emph{primary neon tube}. 
For example, $\set{A_1,A_2,A_3, B_1, B_2, C_1,\intv p, \intv q, \intv r}$ is the set of primary  neon tubes in Figure \ref{fig2}. (Some but not all of the primary neon tubes are lamps.) The rest of the neon tubes, including $\ntm$, $\ntn^-$, $\ntn$, and $\ntn^+$, are secondary.

The following concept is self-explanatory: we say that $\ntn_1,\ntn_2,\ntn_3$ are  \emph{three geometrically consecutive neon tubes} if they belong to the same lamp $I$ and, among the feet of all neon tubes of $I$,  $\Foot{\ntn_i}$ is immediately to the right of $\Foot{\ntn_{i-1}}$ for $i\in\set{2,3}$.
For example, $\ntn^-$, $\ntn$, and $\ntn^+$ are three geometrically consecutive neon tubes in Figure \ref{fig2} but  $\intv q$, $\ntm$, and $\intv  r$ are not.

\begin{lemma}[Three Neon Tubes Lemma]\label{lemma:middlent} 
Let $\ntn_1$, $\ntm=\ntn_2$, and $\ntn_3$ be three consecutive neon tubes of our slim rectangular lattice $L$ such that each of these three neon tubes is secondary. 
Then $\Lbut{\ntm}$, to be defined in \eqref{eq:wwHkhjkbb}, is also a slim rectangular lattice, $\Con{\Lbut{\ntm}}\cong \Con L$, and $|\NTube{\Lbut{\ntm}}| = |\NTube{L}|-1$.
\end{lemma}

\begin{proof}
Clearly, $\ntm$ is an internal neon tube.
Keeping \eqref{eq:cnFnwlHx} in mind, the left and right boundary chains of $L$ are denoted by $\leftb L$ and $\rightb L$, respectively. For $a\in L$, the ideal $\set{x\in L: x\leq a}$ will be denoted by $\ideal a$. Let $\ljc a$ and $\rjc a$ stand for the largest element of $\leftb L\cap\ideal a$ and $\rightb L\cap\ideal a$, respectively. (These acronyms come from \emph{left join coordinate} and \emph{right join coordinate}, respectively; note that both $\ljc a$ and $\rjc a$ belong to $\Jir L\cup\set 0$.)  Let 
\begin{align}
F(\ntm)&:=[\ljc {\Foot{\ntm}}, \Foot{\ntm}] \cup 
[\rjc {\Foot{\ntm}}, \Foot{\ntm}] \cup \set{\Peak \ntm},
\label{eq:wwHkhjkba}
\\
L'&:=\Lbut{\ntm}:=L\setminus F(\ntm).
\label{eq:wwHkhjkbb}
\end{align}
Note that $F(\ntm)$ is a so-called \emph{fork} with top edge $\ntm$; this concept was introduced in Cz\'edli and Schmidt \cite{CzgScht97}. 
We know from \cite[Lemma 20]{CzgScht97} and from the fact that the corners are clearly outside $F(\ntm)$ that $\Lbut{\ntm}$ is a slim rectangular lattice.
For the intervals occurring in \eqref{eq:wwHkhjkba}, we know from  \cite[Lemma 18]{CzgScht97} that,
\begin{equation}
\text{$[\ljc {\Foot{\ntm}}, \Foot{\ntm}]$ and $ 
[\rjc {\Foot{\ntm}}, \Foot{\ntm}]$ are chains.}
\label{eq:mndktrLlnc}
\end{equation}
Furthermore, as it is implicit in, say, Cz\'edli and Schmidt \cite{CzgScht97}, we can assume that the \tbdia-diagram of $\Lbut{\ntm}$ is obtained from that of $L$ in the natural way: we omit the elements of $F(\ntm)$ from the diagram; see how Figure \ref{fig3} is obtained from Figure \ref{fig2}.

We know from, say, Theorem 2.1 and Corollary 2.2 of Cz\'edli, Ozsv\'art and Udvari \cite{CzGozsudv} that for any SPS lattice $K$,  $\length(K)=|\Mir K|$. When passing from $L$ to $\Lbut\ntm$, $\ljc{\Foot\ntm}$ is the only element that we remove from the left boundary chain of $L$. Since the left boundary chain is a maximal chain  and any two finite maximal chains of a semimodular lattice are of the same length,  we obtain that $\length(\Lbut\ntm)=\length(L)-1$. Hence,
$|\NTube {\Lbut\ntm}|= |\Mir {\Lbut\ntm}|= \length(\Lbut\ntm)=\length(L)-1= |\Mir L|-1=|\NTube L|-1$, as required.

Let $L_0$, $L_1$, \dots, $L_k=L$ be a sequence according to \eqref{eq:fztvkTz}. For $i=1,\dots,k$, let $Q_i$ be the lamp that comes to existence when we pass from $L_{i-i}$ to $L_i$; so $Q_i$ is in $\Lamp{L_i}$ and $\Lamp L$ but it is not in $\Lamp{L_{i-1}}$. We know that $\cirrec{Q_i}=H_i$ and $\Lamp L=\set{Q_1,\dots, Q_k}$. 
Assume that $\ntm$ belongs to $Q_j$.  
The \eqref{eq:fztvkTz} sequence for $L':=\Lbut\ntm$ will be denoted by
$L'_0$, $L'_1$, \dots, $L'_k=L'$. We choose this sequence so $L'_i=L_i$ for $i<j$, and their diagrams are also the same. Note that $j>1$ since $Q_j$ is not a boundary lamp.

Let $\ntn$ be a primary neon tube of $Q_j$, and let $\ntn^-$ and $\ntn^+$ be its left neighbour and right neighbour, respectively.  (The case when $\ntn^-$ or $\ntn^+$ does not exists is simpler and will not be detailed.) Since $\ntm$ is secondary and it is sitting between two secondary neon tubes, none of $\ntn$, $\ntn^-$, and $\ntn^+$ is $\ntm$, whereby none of them is removed. 
Hence, none of $\LFloor{\ntn^-}$, $\RFloor{\ntn^-}$, $\LFloor{\ntn}$, $\RFloor{\ntn}$, $\LFloor{\ntn^+}$, and $\RFloor{\ntn^+}$ changes when we remove $\ntm$. These six lines together with the lower boundary of $L$     
 form $\LEA\ntn$ and $\REA\ntn$. Hence
\begin{equation}\left.
\parbox{8cm}{$\LEA\ntn$ and $\REA\ntn$ remain the same for any primary neon tube $\ntn$ of $L_{j}$ when $\ntm$ is removed.}\,\,\right\}
\label{eq:krmpprktmgnm}
\end{equation}
Furthermore, since $\Body {Q_j}$ only depends on its leftmost neon tube and rightmost neon tube, it does not depend on $\ntm$, and so
\begin{equation}
\text{the removal of $\ntm$ does not change $\Body{Q_j}$.}
\label{eq:hglklltnknv}
\end{equation}
We are going to use $\geomeq$
to indicate that two geometrical objects (or two sets of such objects) are exactly the same in a fixed coordinate system of the Euclidean plane $\mathbb R^2$. We know that  $L_{j-1}$ and $L'_{j-1}$ are the same as well as their diagrams. This fact and \eqref{eq:hglklltnknv} gives that
\begin{equation}
\set{\Body I: I\in \Lamp{L'_{j}}}\geomeq \set{\Body I: I\in \Lamp{L_{j}}}
\label{eq:vnwzlFzsghTflda}
\end{equation}
Furthermore, it follows from \eqref{eq:krmpprktmgnm} that 
\begin{align}
&\begin{aligned}
\set{\LEA \ntn&: \ntn\in \Lamp{L'_{j}}\text{ and }\ntn\text{ is primary }} \cr
 &\geomeq \set{\LEA \ntn: \ntn\in \Lamp{L_{j}}\text{ and }\ntn\text{ is primary }}\text{ and}
\end{aligned}
\label{eq:vnwzlFzsghTfldb}\\
&\begin{aligned}
\set{\REA \ntn&: \ntn\in \Lamp{L'_{j}},\text{ and }\ntn\text{ is primary}} \cr
 &\geomeq \set{\REA \ntn: \ntn\in \Lamp{L_{j}}\text{, and }\ntn\text{ is primary}} \text{ and }\cr
\end{aligned}
\label{eq:vnwzlFzsghTfldc}\\
&\text{$L'_j$ is a sublattice (and subdiagram) of $L_j$.}
\label{eq:vnwzlFzsghTfldd}
\end{align}
Trajectories were introduced in Cz\'edli and Schmidt \cite{CzGSchT11}; it is convenient to look into Cz\'edli \cite[Definition 2.13]{CzGlamps} for their definition. We know from the sentence following (2.23) in 
\cite{CzGlamps} that the neon tubes of $L$ are exactly the top edges of the trajectories of $L$. 
We claim that 
\begin{equation}\left.
\parbox{8cm}{if \eqref{eq:vnwzlFzsghTflda}, \eqref{eq:vnwzlFzsghTfldb},  \eqref{eq:vnwzlFzsghTfldc}, and  \eqref{eq:vnwzlFzsghTfldd} hold for some $i$ (in place of $j$) and $i<k$, then they also hold for $i+1$.}
\,\,\right\}
\label{eq:mRnmkgKrm}
\end{equation}
This is almost trivial (at least, visually). Assume that  \eqref{eq:vnwzlFzsghTflda}--\eqref{eq:vnwzlFzsghTfldd} hold for some $i$ (in place of $j$) and $i<k$. To obtain $L_{i+1}$ from $L_i$, we pick a distributive 4-cell $H_{i+1}$ of $L_i$. As a geometric area, $H_{i+1}$ is of the form
$\LEnl{\ntn^\flat}\cap \REnl{\ntn^\sharp}$, where $\ntn^\flat$ is the top edge of the trajectory containing the upper right edge of $H_{i+1}$ while 
$\ntn^\sharp$ is the top edge of the trajectory containing the upper left edge of $H_{i+1}$. Thus, using the validity of  \eqref{eq:vnwzlFzsghTfldb} and  \eqref{eq:vnwzlFzsghTfldc} of $i$, it follows that $H_{i+1}$ is geometrically the same for $L'_i$ as for $L_i$. 
Hence, geometrically exactly the same multifork can be (and is) inserted into $H_{i+1}$ in case of $L'_i$ as in case of $L_i$. In fact, $\cirrec{Q_{j+1}}=\LEnl{\ntn^\flat}\cap \REnl{\ntn^\sharp}$, both in $L'_{j+1}$ and in $L_{j+1}$.
Thus, we conclude \eqref{eq:mRnmkgKrm}.

Since $L'=L'_k$ and $L=L_k$, it follows from  \eqref{eq:mRnmkgKrm} that  \eqref{eq:vnwzlFzsghTflda} holds for  
$L'$ and $L$ (in place of $L'_j$ and $L_j$, respectively).
Now it is clear that $\rhfoot$ is the same for $L'$ as it is for $L$. Hence, we conclude from Lemma \ref{lemma:NTL} 
that $\bigl(\Jir{\Con{L'}},\leq\bigr)\cong \bigl(\Jir{\Con{L}},\leq\bigr)$. By the well-known structure theorem of finite distributive lattices, see, for example, Gr\"atzer \cite[Theorem 107]{r:Gr-LTFound}, $L'\cong L$, as required.
This completes the proof of Lemma \ref{lemma:middlent}.
\end{proof}

Next, we prove the following easy lemma. The \emph{height} of an element $x$ of a finite semimodular lattice will be denoted by $\height (x)$; it is the length of the ideal $\ideal x$.

\begin{lemma}\label{lemma:pwhvknmpczRd}
Let $H$ be distributive $4$-cell of a slim rectangular lattice $L$, and let $L'$ be the (necessarily slim rectangular) lattice that we obtain from $L$ by inserting a $k$-fold multifork into $H$. Then 
$|L'|=|L|+ k\cdot\height({1_H})+ k(k+1)/2$ and $\length(L')=\length(L) +k$.
\end{lemma}

\begin{figure}[ht]\centerline{ \includegraphics[scale=1.0]{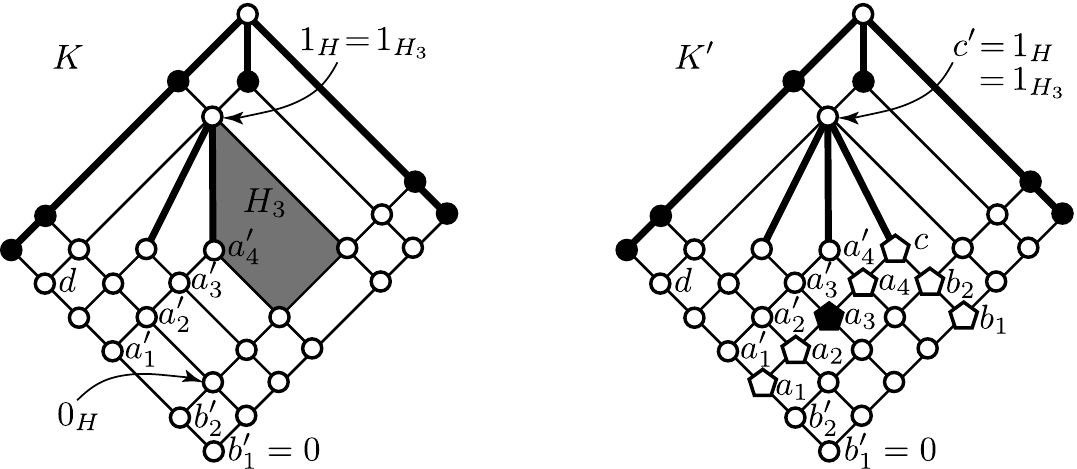}} \caption{Illustrating the proof of Lemma \ref{lemma:pwhvknmpczRd}}\label{fig4}
\end{figure}

\begin{proof}
Since $L$ and $L'$ are semimodular, their lengths are witnessed by their left boundary chains. Hence, the equality $\length(L')=\length(L) +k$ is clear from the definition of adding multiforks. 

The structure theorem based on multiforks, see \eqref{eq:fztvkTz}, is more advantageous than that based on forks, in Cz\'edli and Schmidt \cite[Lemma 22]{CzgScht97}, since while  multiforks are only added to \emph{distributive} 4-cells but this is not so in case if we are only allowed to add forks. However, \emph{in this proof}, it is better to add $k$ forks, on by one, instead of adding a $k$-fold multifork. So we  insert, one by one,  $k$ forks (that is, $1$-fold multiforks $k$ times) into appropriate 4-cells $H_1:=H$, $H_2$, \dots, $H_k$ with the same top $1_{H}$.  
Since 
\[(1+\height(1_H))+(2+\height(1_H))+\dots+(k+\height(1_H))
\]
with $k$ (outer) summands is $ k\cdot\height({1_H})+ k(k+1)/2$,  it suffices to show that for $i=1,\dots,k$,
if $K$ and $K'$ denote the lattice right before and right after inserting the $i$-th multifork into $H_i$, then 
\begin{equation}\left.
\parbox{6.1cm}{$\height_{K'}(1_{H_{i}})=\height_K(1_{H_{i}})+1$ and $|K'|=|K|+\height_{K'}(1_{H_{i}})$.}\,\,\right\}
\label{eq:szRmzhvRkNlt}
\end{equation}
Instead of a formal and lengthy consideration, we use Figure \ref{fig4} to verify \eqref{eq:szRmzhvRkNlt}. 
This figure, where $i=3$, shows how we insert the $i$-th fork into the grey-filled 4-cell $H_3$ of $K$ to obtain $K'$. Implicitly, we will use that 
$1_{H_i}=1_H$ and $\ideal{1_H}$ was distributive before any fork was inserted into $H$.  
With $t=4$ and $s=2$, the new elements are $a_1,\dots,a_t$, $b_1,\dots,b_s$, and $c$; these elements are pentagon-shaped. Assigning an old element $x'$ to each new element $x$, we get a maximal chain 
\[
0=b_1'\prec \dots\prec b_s'\prec a_1'\prec\dots\prec a_t'\prec c'=\height_{K}(1_{H_i})
\]
in $\ideal_K (1_{H_i})$. Hence, the number of new elements is $\height_K(1_{H_{i}})+1$. 
On the other hand, with
$d:=\ljc{1_{H_i}}$, both $\ideal_{\kern -2ptK} d$ and
$[d, 1_{H_i}]$ are chains by Gr\"atzer and Knapp \cite[Lemma 4]{GKn09}. Hence,
$C:= \ideal_{\kern -2ptK} d \cup [d, 1_{H_i}]$ is a maximal chain in $\ideal_{\kern -2ptK} 1_{H_i}$.
Since exactly one element, $a_1$, is added to this chain when we pass from $K$ to $K'$, we obtain that 
$\height_{K'}(_{H_i})= \height_{K}(_{H_i})+1$. Now that we have the first half of \eqref{eq:szRmzhvRkNlt}, the number of new elements is $\height_K(1_{H_{i}})+1 = \height_{K'}(1_{H_{i}})$. We have verified \eqref{eq:szRmzhvRkNlt}, and the proof of Lemma \ref{lemma:pwhvknmpczRd} is complete.
\end{proof}

The following observation only gives a very rough upper bound on the size $|L|$ of $L$ but even such a bound will be sufficient to derive a corollary.

\begin{observation}\label{obs:sKntBsl} Let $D$ be a finite distributive lattice such that $D$ is representable, that is, $D$ is isomorphic to the congruence lattice of a slim rectangular lattice. Then, with the notation $n:=|\Jir D|$, there exists a slim rectangular lattice $L$ such that $\Con L\cong D$,  $\length(L)\leq 3n^2$, and $|L|\leq 9n^4$.
\end{observation}

\begin{proof}
Assume that $L$ is a slim rectangular lattice of minimal size $|L|$ such that $\Con L\cong D$. We know from (ii) of Lemma \ref{lemma:NTL} here, that is, from  Cz\'edli \cite{CzGlamps} that $(\Lamp L;\leq)\cong (\Jir D;\leq)$. Hence $|\Lamp L|=n$. There are at least two boundary lamps (since $L$ is rectangular), so there are at most $n-2$ internal lamps. Observe that if a neon tube $\ntn$ of a lamp $J$ is primary, then $(I,J)\in\nuinfoot$ for some (necessarily internal) lamp $I$.  By Lemma \ref{lemma:middlent},
and the minimality of $|L|$, $J$ cannot have three consecutive secondary neon tubes. Thus, 
\begin{equation}
\text{$J$ has at most $3(n-2)+2 = 3n-4$ neon tubes.}
\label{eq:szLshnsjNc}
\end{equation}
So, taking into account that a boundary lamp has only a single neon tubes,   the total number of neon tubes is at most $n + (n-2)\cdot(3n-4)= 3n^2-9n+8$. Each new lamp with $i$ neon tubes comes to existence by adding an $i$-fold multifork, which increases the length by $i$; see lemma \ref{lemma:pwhvknmpczRd}. This fact, $|\NTube L|\leq 3n^2-9n+8$, and the obvious $\length(L_0)\leq n$ yield that 
$\length(L)\leq n+3n^2-9n+8 \leq 3n^2$, as required.

Finally, to obtain the last inequality stated in the observation, it suffices two show that a slim rectangular lattice (in fact, any SPS lattice) $L$ of length $\ell$ has at most $\ell^2$ elements. We can argue for this easily as follows. By slimness (in the sense of  Cz\'edli and Schmidt \cite{CzGSchT11}), $\Jir L$ is the union of two chains, $C_1$ and $C_2$. By rectangularity and the definition of $\Jir L$, none of $0$ and $1$ is in $C_1$ and $C_2$. So, $|C_1|\leq \ell-1$ and $|C_2|\leq \ell -1$.
Since each element  of $L\setminus\set 0$ is of the form
$c_1\vee c_2$ with $c_1\in C_1$ and $c_2\in C_2$, 
$L$ has at most $1+(\ell-1)^2\leq \ell^2$ elements, indeed. This completes the proof of the observation.
\end{proof}

Recall that slim semimodular lattices are also called SPS lattices.

\begin{corollary}\label{cor:slRms}
There is an algorithm to decide whether a given finite distributive lattice $D$ is isomorphic to the congruence lattice of some SPS lattice $L$; if the answer is affirmative, then the algorithm yields  a slim \emph{rectangular} lattice $L$ such that $D\cong\Con L$. 
\end{corollary}

\begin{proof} Let $n:=|\Jir D|$.
By \eqref{eq:szhGkTnBldR}, it suffices to deal with the question whether there is a slim rectangular lattice $L$ of minimal size such that $\Con L\cong D$. By Observation \ref{obs:sKntBsl}, if such an $L$ exists, then $|L|\leq 9n^4$. Since we can clearly list all the at most $9n^4$-element lattices, we can check which one of them are slim rectangular lattices, and for each such lattice $L$ we can decide whether $\Con L\cong D$, we conclude the corollary.
\end{proof}

\section{Notes on the algorithm}\label{sect:notes-alg}
The algorithm described in the proof of Corollary \ref{cor:slRms} is far from being effective. Even if we do not know if there is a good (better than exponential) algorithm to decide whether there is an SPS lattice with $\Con L \cong \Con D$, we collect some facts about the weakness of the algorithm described in the proof above;
these comments offer some improvements.

\begin{remark}\label{rem:krmnKmlrvzskZlt} 
Instead of constructing all lattices with at most $9n^4$-elements, it is faster (but not fast enough) to list all slim rectangular lattices of length at most $3n^2$; this $3n^2$ comes from Observation \ref{obs:sKntBsl}. But even if we do so, we are still far from a good algorithm. 
Indeed, we know from  Cz\'edli, D\'ek\'any, Gyenizse, and Kulin \cite{CzGDTGyTKJ} that the number of slim rectangular lattices of length $k$ is asymptotically 
$(k-2)!\cdot e^2/2$, where $e$ is the famous mathematical constant $\lim_{n\to\infty}(1+1/n)^n\approx 2.718\,281\,828$. Thus, there are about 
\begin{equation}
x(n):=(e^2/2)\cdot\sum_{k=2}^{3n^2} (k-2)!
\label{eq:msGFrmkohnlPjt}
\end{equation} 
many slim rectangular lattices to verify, and we could hardly verify that many. Indeed, say,
\begin{equation}
x(5)\approx 0.167\cdot 10^{107}\text{ and }x(9)\approx 0.3637\cdot 10^{472},
\label{eq:gyszmlT}
\end{equation}
indicate that even with the help of a computer, the method given so far is not enough to  decide whether $D$ with $|\Jir D|$ can be represented in the required way.  
\end{remark}

\begin{remark}\label{rem:hkmsGrmbn} 
Since our purpose was to give short proofs, 
the estimates $3n^2$ and $9n^4$ in Observation \ref{obs:sKntBsl} are far from being optimal, because of several reasons. First, our computation was based on the Three Neon Tubes Lemma, that is, Lemma \ref{lemma:middlent}, although the ``Two Neon Tubes Lemma'' (asserting that if there are two consecutive secondary neon tubes, then one of them can be removed) seems  also be true. (The ``Two Neon Tubes Lemma''  would require a more complicated and much longer proof than Lemma \ref{lemma:middlent} while not leading to a feasible algorithm, so we neither prove nor use this lemma.) Second, \eqref{eq:szLshnsjNc} is a rather
weak estimate for most $J\in\Lamp L$; indeed, if 
$I\in\Lamp L$ witnesses that a neon tube $\ntn$ of $J$ is primary, that is, if $
\Foot I$ belongs to $\GInt{\LEA\ntn}\cup\GInt{\REA\ntn}$, then $I<J$. So if $\ideal J$, understood in $(\Lamp L,\leq)$, is a small set, then only few neon tubes of $J$ can be primary, and Lemma \ref{lemma:middlent}  yields that $J$ only has few neon tubes. Furthermore, we know from Lemma \ref{lemma:NTL} that each lower cover of $J$ 
is illuminated by a primary neon tube of $J$, but the rest of lamps belonging to $\ideal J$ need not be. 
\end{remark}

\begin{remark}\label{rem:mnSkwrsZgds}
Even if we used the ideas above to improved the algorithm given Corollary \ref{cor:slRms}, it would not be feasible enough. One of the reasons is that if we construct all slim rectangular lattices of a given length $k$ \emph{without} keeping the poset $\Jir L$ in mind, then 
approximately $x(k)$ many lattices, so too many lattices should be constructed; see \eqref{eq:msGFrmkohnlPjt}. 
\end{remark}

\begin{figure}[ht] \centerline{ \includegraphics[scale=1.0]{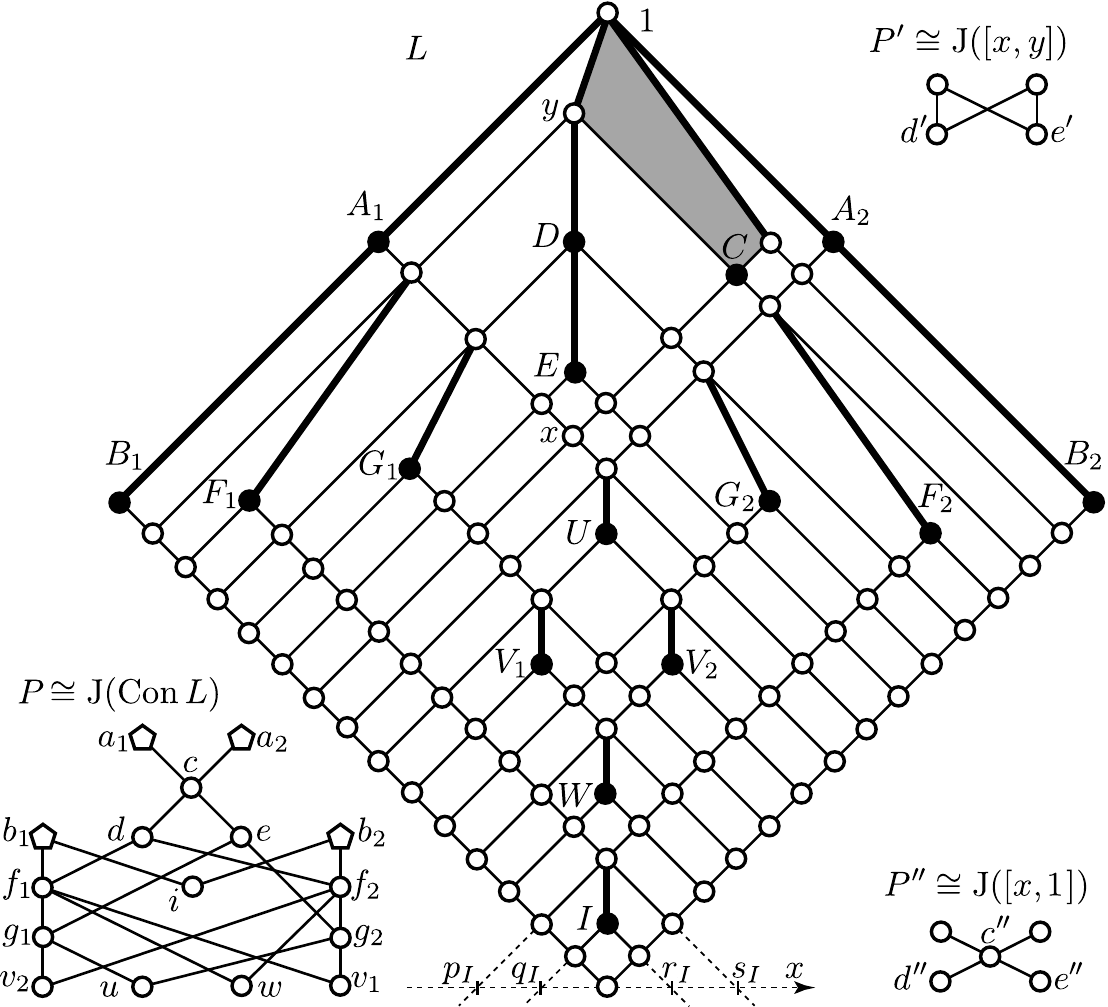}}\caption{The secondary neon tube of $C$ cannot be omitted}\label{fig5}
\end{figure}

\begin{remark}\label{rem:shzLbB} 
A secondary neon tube can quite frequently be omitted but not always. We present some examples. 
In the first example, let $Y$ denote  the four-element ``Y-shaped'' poset $\set{0,c,a,b}$ such that $0\prec c$, 
$c\prec a$, and $c\prec b$ is a full list of coverings. 
Then  $c$ only has one lower cover,  there exists a slim rectangular lattice $L$ such that 
$\Jir{\Con L}\cong Y$, but for every such $L$, the lamp $C\in\Lamp L$ corresponding to $c\in Y$ necessarily has at least one secondary neon tube.
\cskjsg{Modify the $H$-part of \ref{fig5} to show an example.}

The second example, given by Figure \ref{fig5}, 
 shows how to represent the poset $P$ on the bottom left as 
$\Jir{\Con L}$ where $L$ is the slim rectangular lattices drawn in the middle. As usually in the paper, $x\in P$ is represented by $X\in\Lamp L$, for any letter $x$. Note that $\Body C$ is grey-filled in the figure. Observe that $c$ has two lower covers (so more than $c$ in the first example) but $C$ still has only two neon tube. Furthermore, the neon tube of $C$ on the left is secondary.

Figure \ref{fig5} also shows how the represent $P'$ and $P''$, drawn on the right, by certain intervals of $L$; these intervals are slim rectangular lattices. It would be easy to construct similar examples with arbitrary many minimal elements while keeping the subposet of $P'$ or $P''$ formed by the non-minimal elements unchanged.

As another example, we mention that  $f_1$ in  Figure \ref{fig5} has three lower covers but $F_1\in\Lamp L$ only has one neon tube. 
\end{remark}

Remark \ref{rem:shzLbB} is our excuse that we do not try to determine the  minimal number of  neon tubes of a lamp $X$ representing an element $x$ of a poset.

\section{An algorithm based on illuminated sets}\label{sect:illum-alg}
In this section, we are going to point out that 
it is frequently advantageous to base our investigation on $\EnS L$  rather than $\Lamp L$; in this way we can reduce many problems about $\Con L$  to combinatorial geometric problems about illuminated sets.  Furthermore, Lemma \ref{lemma:LsrTn} of this section, which is formulated both for lamps and for illuminated sets, will be used in subsequent sections.

The paragraph we commence here is to warn the reader. The algorithm described in this section is \emph{much} more complicated than the one described by (the proof of) Corollary \ref{cor:slRms}. 
Indeed, while one can understand in a second that checking all lattices $L$ with at most $9\cdot |\Jir D|^4$ elements is  an algorithm, it is far from being conspicuous how to use the algorithm we only roughly describe here.  A conjecture right after Lemma \ref{lemma:mrclgnsBfZk} would result in some improvement but this conjecture is not proved.  Admittedly, a more detailed and elaborated algorithm in a much longer paper could be possible.

However, in spite of the non-appetizing message carried by the previous paragraph, this is the algorithm what we can use for small lattice. Experience shows that for 
$n:=|\Jir D|\leq 5$ (almost) surely and with good chance even for
$n=9$ we can decide (without computer!) whether $D$ is representable as the congruence lattice of an SPS lattice. On the other hand, \eqref{eq:gyszmlT} indicates that even if we use computers, the easy-to-understand algorithm of the previous section is not sufficient for the same purpose.

The ideas of the algorithm described here have already been used in proofs and they will hopefully be used in future proofs.

Note that the theory of lamps and the algorithm
mutually influence each other. If we put more theory into the algorithm, e.g., if we could continue the list $(\#1)$, $(\#2)$
in Definition \ref{def:lvntmgvhlmzk}, then the algorithm would become better, that is, faster. Conversely, some ideas of the algorithm  have already been used in discovering facts and proving them,  and a better algorithm could lead to new discoveries and their proofs. Implicitly, this is happening here in Sections \ref{sect:lampslemmas}--\ref{sect:CMP}.

For a lamp $I\in\Lamp L$, where $L$ is from \eqref{eq:cnFnwlHx}, the illuminated set $\Enl I$ defined in \eqref{eq:vmhlTkvtc} is a geometric area in the plane. As in Cz\'edli and Gr\"atzer \cite[Definition 4.1(ii) and Figure 2]{CzGGG3p3c}, $\Enl I$ can be described by its \emph{coordinate quadruple} $(p_I, q_I,r_I,s_I)$, which belongs to $\mathbb R^4$; see also the multi-purposed Figure \ref{fig5} here. Let
\begin{equation}
\EnS L:=\set{\Enl I: I\in\Lamp L}.
\label{eq:snfrLkrRw}
\end{equation}
With reference to \eqref{eq:floorsxTd}, it is clear that for $I\in \Lamp L$,  $\LFloor I$, $\RFloor I$, and $\Floor I$
are determined by $H:=\Enl I$. Hence $\Foot I$, which is the intersection point of $\LFloor I$ and $\RFloor I$, is also determined by $\Enl I$. So is $\Peak I$.  These fact allow us to write $\LFloor H$, $\RFloor H$, $\Floor H$,  $\Foot H$, $\Peak H$, and  $(p_H, q_H,r_H,s_H)$.
Clearly, $I$ is a boundary lamp if and only if $H=\Enl I$ is a stripe of normal slope, and $I$ is an internal lamp if and only if $H$ is an ``A-shape'' (that is, a ``V-shape'' turned upside down).
This allows us to say that $H$ is a \emph{boundary illuminated set} or an \emph{internal illuminated set}, respectively. 
Motivated by  Cz\'edli \cite[Definition 2.9(vi)--(vii)]{CzGlamps} and
Definition \ref{defrhRhs}\eqref{defrhRhsb}, for $H_1, H_2\in \EnS L$, we define
\begin{align}
&(H_1,H_2)\in\rhfoot\defiff \Foot{H_1}\in H_2 \text{ and $H_1$ is internal,}
\label{eq:shrknnya}\\
&(H_1,H_2)\in\rhinfoot\defiff \Foot{H_1}\in \GInt{H_2}\text{ and}\label{eq:shrknnyb}\\
&\text{let $\leq$ be the reflexive transitive closure of $\rhinfoot$. }
\label{eq:shrknnyc}
\end{align}
Note that the condition  $\Foot{H_1}\in \GInt{H_2}$ in 
\eqref{eq:shrknnyb} automatically implies that $H_1$ is an internal illuminated set. The following lemma follows trivially from 
Cz\'edli \cite[Lemma 2.11]{CzGlamps}.

\begin{lemma} \label{lemma:lTszhmhDz}
Let $L$ be as in \eqref{eq:floorsxTd}. Then, on the set $\EnS L$, the relation $\rhfoot$ defined in \eqref{eq:shrknnya} is the same as $\rhinfoot$ defined in \eqref{eq:shrknnyb}. Furthermore, with ``$\leq$'' defined in \eqref{eq:shrknnyc}, $(\EnS L;\leq)$ is a poset isomorphic to
$(\Lamp L;\leq)$ and also to  $(\Jir{\Con L};\leq)$.
\end{lemma}

\begin{figure}[ht]\centerline{ \includegraphics[scale=1.0]{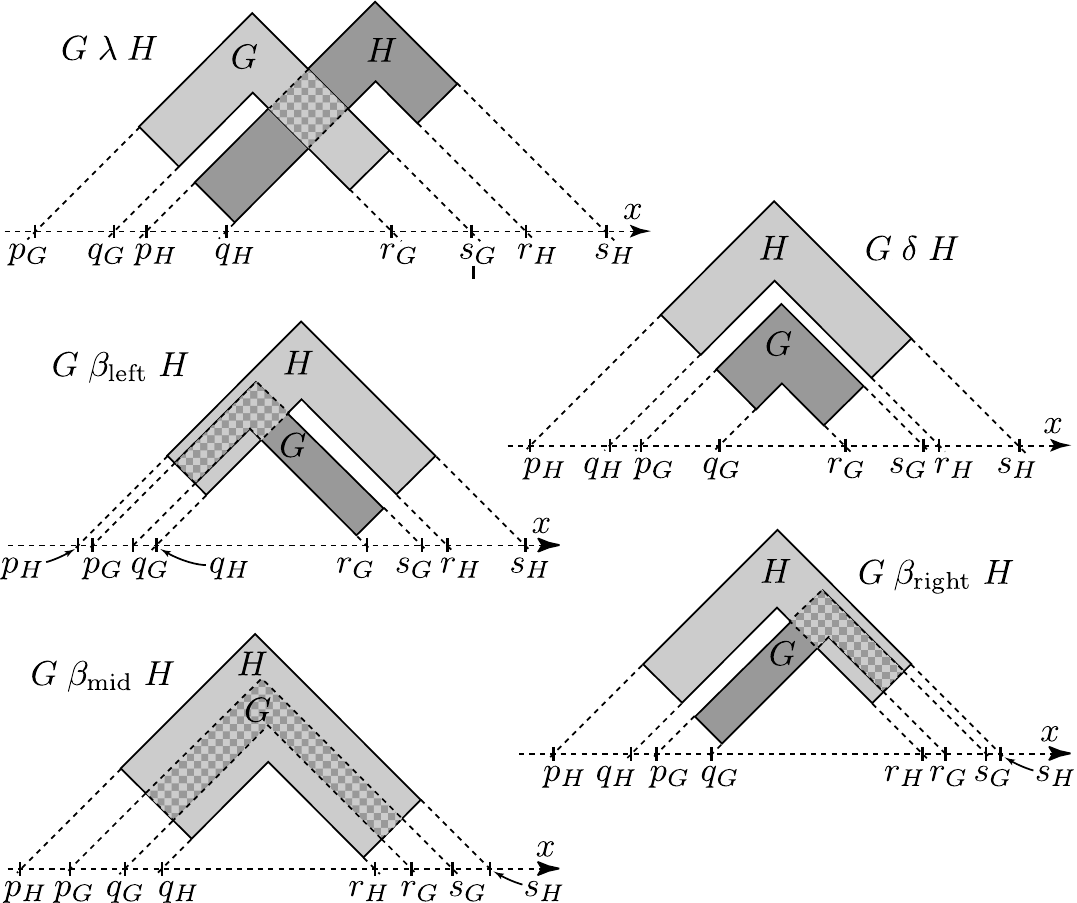}} \caption{Possible positions of two illuminated sets in the plane}\label{fig6}
\end{figure}

\begin{definition}\label{def:gmpzckrbn}
Extending Cz\'edli and Gr\"atzer \cite[Definition 4.1.(iii)]{CzGGG3p3c} and Cz\'edli \cite[(4.1)]{CzGlamps}, we define the following relations for $G,H\in\EnS L$; see Figure \ref{fig6} for illustrations.
\begin{itemize}
\item $G\rlambda H$, that is, $G$ is \emph{to the left of} $H$ if $q_G\leq p_H$ and $s_G\leq r_H$; 
\item $G\rdelta H$, that is, $G$ is \emph{geometrically under} $H$ if $q_H\leq p_G$ and $s_G\leq r_H$;
\item $G\rbm H$ if $p_H<p_G<q_G<q_H<r_H<r_G<s_G<s_H$;
\item $G\rbl H$ if $p_H\leq p_G<q_G<q_H$, $s_G\leq r_H$, and $G$ is internal;
\item $G\rbr H$ if $q_H\leq p_G$, $r_H< r_G<s_G\leq s_H$, and $G$ is internal.
\end{itemize}
Furthermore, for $I,J\in\Lamp L$, let $I\rlambda J$, $I\rdelta J$, $I\rbm J$, $I\rbl J$, and $I\rbr J$ mean that 
$\Enl I\rlambda \Enl J$, $\Enl I\rdelta \Enl J$, $\Enl I\rbm \Enl J$, $\Enl I\rbl \Enl J$, and $\Enl I\rbr \Enl J$, respectively.
\end{definition}

The notations $\rlambda$ and $\rdelta$ come from ``Left'' and ``unDer'', respectively. Clearly,
\begin{equation}
\text{$\rlambda$ and $\rdelta$ are irreflexive and transitive relations.}
\label{eq:wknhTtPkrzrk}
\end{equation}

\begin{lemma}\label{lemma:LsrTn}
Let $L$ be a slim rectangular lattice. Let $I$ and $J$ be either two distinct members of $\Lamp L$ or two distinct members of $\EnS L$. Then exactly one of the following ten alternatives hold.
\begin{enumeratei}
\item\label{lemma:LsrTna} $I\rlambda J$,
\item\label{lemma:LsrTnb} $J\rlambda I$ (that is, $I$ is \emph{to the right of} $J$, which is sometimes denoted by $I\mathrel \rho J$.)
\item\label{lemma:LsrTnc} $I\rdelta J$,
\item\label{lemma:LsrTnd} $J\rdelta I$,
\item\label{lemma:LsrTne} $I\rbm J$,
\item\label{lemma:LsrTnf} $J\rbm L$
\item\label{lemma:LsrTng} $I\rbl J$,
\item\label{lemma:LsrTnh} $J\rbl I$,
\item\label{lemma:LsrTni} $I\rbr J$,
\item\label{lemma:LsrTnj} $J\rbr I$.
\end{enumeratei}
Furthermore, if $I$ and $J$ are incomparable in the poset $(\Lamp L;\leq)$  or in the poset $(\EnS L;\leq)$, then 
the first four options are only possible.
\end{lemma}

\begin{proof} The internal illuminated sets are ``A-shapes'' (i.e., ``V-shapes'' turned upside down) with thickness or stripes. Those possible mutual geometric positions of two A-shapes that are not listed in the lemma are ruled out by Lemma 3.8 of \cite{CzGlamps}.
If one of \eqref{lemma:LsrTne},\dots,\eqref{lemma:LsrTnj} holds, then $(I,J)$ or $(J,I)$ belongs to $\rhfoot$ and $I$ and $J$ are comparable by Lemma \ref{lemma:NTL}. Therefore, only \eqref{lemma:LsrTna}, \dots, \eqref{lemma:LsrTnd} are allowed if $I$ and $J$ are incomparable.
\end{proof}

\begin{definition}\label{def:lvntmgvhlmzk}
Next, assume that we partition a rectangle into finitely many stripes by lines of slope $3\pi/4$; these stripes will be called \emph{abstract left boundary illuminated sets}. Similarly, 
we partition the same rectangle into finitely many stripes by lines of slope $\pi/4$ to obtain the \emph{abstract right boundary illuminated sets}.  Then we add finitely many A-shapes called  \emph{abstract internal illuminated sets} such that 
\begin{enumerate}
\item[$(\#1)$] Lemma \ref{lemma:LsrTn} holds for these abstract sets, 
\item[$(\#2)$] For any abstract boundary illuminated set $Z$, the condition formulated in  (4.3) (and Lemma 3.9) of Cz\'edli \cite{CzGlamps} is satisfied.
\end{enumerate}
Then we say the our finite collection of abstract illuminated sets is an \emph{abstract illuminated system}. 

Two such systems are called \emph{similar} if there is a bijective correspondence $\phi$ between them such that both $\phi$ and $\phi^{-1}$ preserve each of the  five relations described in Definition \ref{def:gmpzckrbn}.
\end{definition}

For $I\in\Lamp L$, $\Floor I$, $\Roof I$, $\Foot I$ and $\Peak I$ are determined by $\Enl I$. This allows us to define these objects, in a natural way, for $H\in\EnS L$. Then, also, $\rhfoot$ and $\rhinfoot$ are defined on $\EnS L$ and they are equal. (If their equality is not a consequence of definitions, then it should  be added to Definition \ref{def:lvntmgvhlmzk}  as $(\#3)$.)

\begin{definition}\label{def:slwZskms}
An abstract illuminated system $\mathcal S$ is also a \emph{poset}  $\mathcal S:=(\mathcal S;\leq)$ where ``$\leq$'' is the  
reflexive transitive closure of \eqref{eq:shrknnyc}.
\end{definition}

Comparing  \eqref{eq:shrknnyc} to Definition \ref{def:slwZskms}
and using \eqref{eq:szhGkTnBldR} and Lemma \ref{lemma:lTszhmhDz}, we obtain the validity of the following lemma.

\begin{lemma}\label{lemma:mrclgnsBfZk} Let $D$ be a finite distributive lattice. If $D$ is representable as $\Con K$ for an SPS lattice $K$, then 
\begin{enumeratei}
\item \label{lemma:mrclgnsBfZka}
there is an abstract illuminated system $(\mathcal S;\leq)$ isomorphic to $(\Jir D;\leq)$ and
\item there is a slim rectangular lattice $L$ such that $(\EnS L;\leq)$ is similar to the above-mentioned $(\mathcal S;\leq)$.
\end{enumeratei}
\end{lemma}

We conjecture that Condition \eqref{lemma:mrclgnsBfZka} of this lemma is not only a necessary but also a sufficient condition of the representability of $D$. If this is so, then 
the algorithm below becomes faster. But even though we do not prove this conjecture, Lemma \ref{lemma:mrclgnsBfZk} together
 with other known facts lead to the following algorithm.

\begin{algorithm} \label{alg:gqnmpn}
Assume that $D$ is a finite distributive lattice to be represented as the congruence lattice of an SPS (=slim semimodular) lattice. By \eqref{eq:szhGkTnBldR}, we can assume that this SPS lattice is a slim rectangular lattice $L$; see also \eqref{eq:cnFnwlHx} (If this $L$ exists, then the algorithm will construct it.) Let $n:=|\Jir D|$. 
We are going to find an abstract illuminated system $\mathcal S$ 
such that $(\mathcal S;\leq)$ isomorphic to $(\Jir D;\leq)$. 
Even if there are continuously many $n$-element abstract illuminated systems, we are only interested in $\mathcal S$ up to similarity. Any $n$-element abstract illuminated system is
described by $n$  coordinate quadruples. Although the entries of these quadruples are real numbers (of which there are too many), the system up to similarity is determined by how these entries are ordered. Therefore, we can fix a $4n$-element set 
$U$ of real numbers such that each of the $n$ coordinate tuples belongs to $U^4$. Note, however, that when we use the algorithm (without computers), then we draw figures rather than paying attention to any $U$; $U$ is only mentioned here because its finiteness indicates that we are describing an \emph{algorithm}.

We only list those abstract illuminated systems that, according to our theoretical knowledge, might be isomorphic to $(\Jir D;\leq)$. First, by Cz\'edli \cite[Lemma 3.2]{CzGlamps}, there should be exactly $|\Max{\Jir D}|$ many boundary illuminated sets (that is, stripes) since they correspond to the maximal elements of $\Jir D$. Second, when deciding which of these $|\Max{\Jir D}|$ many boundary illuminated sets should be on left and which on the right, we take the Bipartite Maximal Elements Property of Cz\'edli \cite[Corollary 3.4]{CzGlamps} into account. 
Either in the meantime or at the beginning, it is reasonable to check if $(\Jir D;\leq)$ satisfies the seven previously known properties; see  
Cz\'edli \cite{CzGlamps} and Cz\'edli and Gr\"atzer \cite{CzGGG3p3c} where these properties are (first) proved or cited  from Gr\"atzer \cite{gG16} and \cite{gG20}. The properties occurring in the present paper are also useful as well as the known properties of lamps (translated to illuminated sets) are also useful since they exclude lots of case; see,  Section \ref{sect:lampslemmas} for some properties of lamps.

After parsing all the cases ``permitted by known properties'', we can decide if there exists an abstract illuminated system $\mathcal S$  such that $(\mathcal S;\leq)$ is isomorphic to $(\Jir D;\leq)$. If such an $\mathcal S$ does not exists, then $D$ cannot be represented in the required way and the algorithm concludes with ``no''. If $\mathcal S$ exists and the conjecture right after Lemma \ref{lemma:mrclgnsBfZk} is true, then the algorithm concludes with a positive answer. 

If  $\mathcal S$ exists  but either we do not know whether the conjecture is true or we need to construct $D$, then we can do the following. Based on $\mathcal S$ (and slightly modifying it to a similar system from time to time when we bump into obstacles), we try to construct $D$; indeed, $\mathcal S$ serves as an outline and a bird's-eye view of $D$.
If we succeed, the algorithm concludes with ``yes'' and $L$ is also found. Otherwise, we try to construct $D$ from another $\mathcal S$. If, after constructing all $\mathcal S$ with $(\mathcal S,\leq)\cong(\Jir D;\leq)$ but failing to construct $L$ from them, the algorithm yields a negative answer.
\end{algorithm}

\section{Some easy lemmas about lamps}\label{sect:lampslemmas}
In this section, as a preparation for Sections \ref{sect:CTF} and \ref{sect:CMP}, we prove some easy statements.

\begin{lemma}\label{lemma:lttkNgPfD}
If $L$ is as in \eqref{eq:cnFnwlHx}, $I,J\in\Lamp L$, and $I$ is geometrically under $I$ (in notation, $I\rdelta J$), then $I\nprec J$ in $\Lamp L$. Equivalently, if 
and $I\prec J$ in $\Lamp L$, then  $I\rdelta J$ cannot hold. 
\end{lemma}

\begin{proof}
Suppose the contrary. Then $I\rdelta J$ and, by Lemma \ref{lemma:NTL}, $(I,J)\in\rhbody$. Since $I\rdelta J$, 
$\Enl I$ and $\Enl J$ are sufficiently disjoint in the sense of (3,4) of Cz\'edli \cite{CzGlamps}, contradicting 
 $(I,J)\in\rhbody$.
\end{proof}

For $U\in\Lamp L$, let $\gideal\Roof U $ denote the set of those geometric points of 
the full geometric rectangle (of the \tbdia-diagram of $L$) that are on or below $\Roof U$. More precisely, a geometric point $(x,y)$ (given in the usual coordinate system)
of the full geometric rectangle belongs to $\gideal\Roof U$ if and only if $(x,y')\in\Roof U$ for some $y'$ such that $y'\geq y$. 

\begin{lemma}\label{lemma:wszbLlw} If $L$ is from \eqref{eq:cnFnwlHx} and $I<J$ holds in $\Lamp L$, then
$\gideal\Roof I\subseteq \gideal\Roof J$.
\end{lemma}

\begin{proof}If $I\prec J$, then $(I,J)\in\rhbody$ by  Lemma \ref{lemma:NTL}, whence $\Body I\subseteq \gideal\Roof J$ gives the required inclusion $\gideal\Roof I\subseteq \gideal\Roof J$. Otherwise, the inclusion follows from its just-mentioned particular case by transitivity.
\end{proof}

Next, we prove the following lemma; the conjunction of this lemma with Lemma \ref{lemma:LsrTn} is stronger than  Cz\'edli and Gr\"atzer \cite[Lemma \ref{lemma:LsrTn}]{CzGGG3p3c}.

\begin{lemma}\label{lemma:stngrrczg}
For $L$ from \eqref{eq:cnFnwlHx} and $I,J,K\in\Lamp L$,  if $I\rdelta J$ and $K<I$, then $K\nprec J$. That is, if $I$ is geometrically under $J$, then no element of the principal ideal $\ideal I$ is covered by $J$ in $\Lamp L$. 
\end{lemma}

\begin{proof}
Assume that $I\rdelta J$ and $K\leq I$.
Lemma \ref{lemma:wszbLlw} gives that 
$\gideal\Roof K\subseteq \gideal \Roof I$. Hence, $\Foot K\in\gideal\Roof I$. Actually, $\Foot K\in\GInt{\gideal\Roof I}$ since at least one precipitous edge going upwards starts at or above $\Foot K$. Since $I\rdelta J$,  
$\GInt{\gideal\Roof I}\cap \GInt{\Enl J}=\emptyset$. Hence, $\Foot K\notin \GInt{\Enl J}$, that is, $(K,J)\notin\rhinfoot$. Therefore, the required $K\nprec J$ follows by Lemma \ref{lemma:NTL}.
\end{proof}

\begin{lemma}\label{lemma:snknwkTh}
If $L$ is from  \eqref{eq:cnFnwlHx}, $I,J,K\in\Lamp L$,
$J\rdelta K$, and $I\leq J$, then $I\rdelta K$.
\end{lemma}

\begin{proof} Apply Lemma \ref{lemma:wszbLlw}.
\end{proof}

Yet we state another easy lemma. For an illustration, see Figure 8 in Cz\'edli and Gr\"atzer \cite{CzGGG3p3c}.

\begin{lemma}\label{lemma:xlhmQjlRn}
Assume that  $L$ is from \eqref{eq:cnFnwlHx}, $A_0$, $A_1$, $A_2$ and $B_1$ are from $\Lamp L$, 
$A_0\rlambda A_1 \rlambda A_2$, $B_1\prec A_0$, and $B_1\prec A_2$. Then $B_1\rdelta A_1$. 
\end{lemma}

\begin{proof} By Lemma \ref{lemma:NTL}, $(B_1,A_0), (B_1,A_2)\in\rhbody$. Hence $\Body{B_1}\subseteq \Enl {A_0}\cap \Enl {A_2}$, see Cz\'edli and Gr\"atzer \cite[Figure 8]{CzGGG3p3c}, and we obtain that $B_1\rdelta A_1$. 
\end{proof}

\begin{figure}[ht]\centerline{ \includegraphics[scale=1.0]{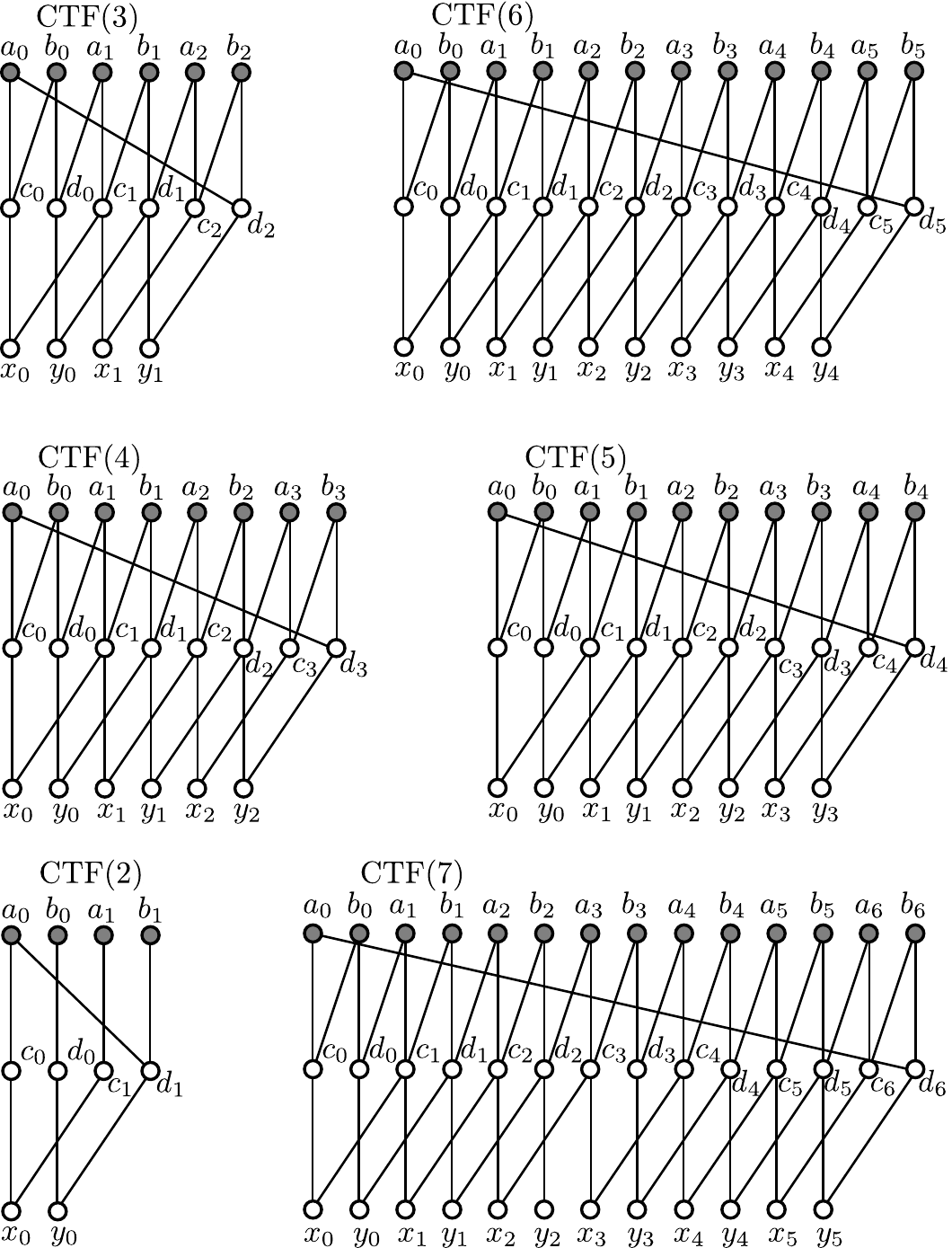}} \caption{The posets $\CTF 2$, \dots, $\CTF7$}\label{fig7}
\end{figure}

\section{An infinite family of new properties of congruence lattices of SPS lattices}\label{sect:CTF}
For an integer $n\geq 2$, we define the poset \emph{Crown with Two Fences of order} $n$, in notation $\CTF n$ as follows; see also Figure \ref{fig7}.
The elements of $\CTF n$ are $a_0$, $a_1$, \dots, $a_{n-1}$, 
$b_0$, $b_1$, \dots, $b_{n-1}$, $c_0$, $c_1$, \dots, $c_{n-1}$,  $d_0$, $d_1$, \dots, $d_{n-1}$, $x_0$, $x_1$, \dots, $x_{n-2}$, and $y_0$, $y_1$, \dots, $y_{n-2}$; they are pairwise distinct.  The edges (in other words, the prime intervals) are as follows, the arithmetic in the subscripts is understood on $\ZZ n$, that is, modulo $n$:   $c_i\prec a_i$ and $c_i\prec b_i$ for $i\in\ZZ n$,  $d_i\prec b_i$ and $d_i\prec a_{i+1}$ for $i\in \ZZ n$, $x_j\prec c_j$ and $x_j\prec c_{j+1}$ for $j\in\set{0,1,\dots,n-2}$, and $y_j\prec d_j$ and $y_j\prec d_{j+1}$ for 
$j\in\set{0,1,\dots,n-2}$. In the figure, the maximal elements of $\CTF n$ are grey-filled.

For posets $P_1$, $P_2$ and a map (ALSO KNOWN AS function)  
$\phi\colon P_1\to P_2$, we say that $\phi$ is an \emph{embedding} if 
\begin{equation}
\text{for any $x,y\in P_1$, \ $x\leq y$ in $P_1$ if and only if $\phi(x)\leq \phi(y)$ in $P_2$.}
\label{eq:smgstrszp}
\end{equation}
If $\phi\colon P_1\to P_2$ is an embedding  and, \emph{in addition,} 
\begin{equation}
\text{for any $x,y\in P_1$, \ $x\prec y$ in $P_1$ if and only if $\phi(x)\prec \phi(y)$ in $P_2$,}
\label{eq:krklwjskrs}
\end{equation}
then $\phi$ is called a \emph{cover-preserving embedding}.
Finally, if $\phi\colon P_1\to P_2$ is an embedding such that $\phi(x)$ is a maximal element of $P_2$ for every maximal element $x\in P_1$, then $\phi$ is a \emph{maximum-preserving} embedding. 
By an SPS lattice we still mean a slim semimodular lattice.

\begin{definition} For an integer $n\geq 2$ and a poset $P$, we say that $P$ satisfies the \emph{$\CTF n$-property } if there exists no cover-preserving embedding $\phi\colon \CTF n\to P$ that is also maximum-preserving. 
\end{definition}

\begin{theorem}\label{thm:fncZ}
For every integer $n\geq 2$ and any SPS lattice $K$, $\Jir{\Con K}$ satisfies the $\CTF n$-property. 
\end{theorem}

\begin{proof} By way of contradiction, suppose that the theorem fails for some $n\geq 2$. By Lemma \ref{lemma:NTL}\eqref{lemma:NTLb}, $(\Jir L;\leq)\cong (\Lamp L;\leq)$ for a slim rectangular lattice $L$. Hence, there is a maximum-preserving and cover-preserving embedding $\phi \colon \CTF n\to \Lamp L$; for $x\in \CTF n$, we denote $\phi(x)$ by the corresponding capital letter, $X$. By left-right symmetry and the Bipartite Maximal Elements Property, see  Cz\'edli \cite[Lemma 3.4]{CzGlamps}, we can assume that $A_0,\dots, A_{n-1}$ are left boundary lamps while 
$B_0,\dots, B_{n-1}$ are right boundary lamps. (There can be other boundary lamps but they are not $\phi$-images and cause no trouble.) Let $A_i\rblambda A_j$ mean that $A_i$ is to the left of $A_j$ on the left upper boundary of $L$. Similarly, $B_i\rblambda B_j$ means that $B_i$ is to the left of $B_j$ on the right upper boundary of $L$. We know from \eqref{lemma:NTLa} and \eqref{lemma:NTLc} of Lemma \ref{lemma:NTL} that 
\begin{equation}
\text{if $x\prec y$ in $\CTF n$, then $(X,Y)\in \rhbody$  and so $\Body X\subseteq \Enl Y$.}
\label{eq:nzsmzkmGlKmn}
\end{equation}
We claim that
\begin{equation}
\text{for $i\in\set{0,1,\dots,n-2}$, if $A_i\rblambda A_{i+1}$, then 
$B_i\rblambda B_{i+1}$}. 
\label{eq:csnkrnkcsJr}
\end{equation}
We prove this by way of contradiction. Suppose that  $A_i\rblambda A_{i+1}$ holds but $B_i\rblambda B_{i+1}$ fails. Then  
$B_{i+1}\rblambda B_{i}$. Since we know from  \eqref{eq:nzsmzkmGlKmn} that $\Body{C_{i}}\in\Enl{A_{i}} \cap \Enl{B_{i}}$ and  $\Body{C_{i+1}}\in\Enl{A_{i+1}} \cap \Enl{B_{i+1}}$, we obtain that $C_i\rdelta C_{i+1}$. In virtue of Lemma \ref{lemma:stngrrczg}, $C_i\rdelta C_{i+1}$ and $X_i< C_i$ implies that $X_i\nprec C_{i+1}$. This is a contradiction since $\phi$ is cover-preserving. We have proved  \eqref{eq:csnkrnkcsJr}.

Next, we claim that 
\begin{equation}
\text{for $i\in\set{0,1,\dots,n-2}$, if $B_i\rblambda B_{i+1}$, then 
$A_{i+1}\rblambda A_{i+2}$}; 
\label{eq:nhnmgdhGmtDd}
\end{equation}
here $i+2$ is understood in $\ZZ n$, that is, $(n-2)+2=0$.  To prove \eqref{eq:nhnmgdhGmtDd} by way of contradiction, suppose that  $B_i\rblambda B_{i+1}$ but $A_{i+2}\rblambda A_{i+1}$. Then, similarly to the argument given for \eqref{eq:csnkrnkcsJr}, 
 \eqref{eq:nzsmzkmGlKmn} yields that $D_{i+1}\rdelta D_i$.
Using $D_{i+1}\rdelta D_i$, $Y_i< D_{i+1}$, and $Y_i\prec D_{i}$, Lemma \ref{lemma:stngrrczg} gives a contradiction and proves \eqref{eq:nhnmgdhGmtDd}.

Clearly, either $A_0\rblambda A_1$ or $A_1\rblambda A_0$ (that is, 
$A_0\mathrel\rho A_1$, see Lemma \ref{lemma:LsrTn}\eqref{lemma:LsrTnb}). By symmetry,  \eqref{eq:csnkrnkcsJr}  and \eqref{eq:nhnmgdhGmtDd} also hold for $\mathrel\rho$. Thus, we can assume that $A_0\rblambda A_1$, and we can argue as follows; when referencing \eqref{eq:csnkrnkcsJr} or \eqref{eq:nhnmgdhGmtDd} over implication signs, the value of $i$ will be indicated. We obtain that
\begin{align}\left.
\begin{aligned}
A_0\rblambda A_1 &
\overset{(\ref{eq:csnkrnkcsJr},i=0)}\Longrightarrow
B_0\rblambda B_1 \overset{(\ref{eq:nhnmgdhGmtDd},i=0)}\Longrightarrow
A_1\rblambda A_2 \overset{(\ref{eq:csnkrnkcsJr},i=1)} \Longrightarrow 
B_1\rblambda B_2 \cr
& \overset{(\ref{eq:nhnmgdhGmtDd},i=1)}\Longrightarrow A_2\rblambda A_3 \overset{(\ref{eq:csnkrnkcsJr},i=2)}\Longrightarrow B_2\rblambda B_3 \overset{(\ref{eq:nhnmgdhGmtDd},i=3)}\Longrightarrow A_3\rblambda A_4\text{ } \dots \cr
& \dots \text{ } A_{n-2}\rblambda A_{n-1} \cr
&  \overset{(\ref{eq:csnkrnkcsJr},i=n-2)}\Longrightarrow B_{n-2}\rblambda B_{n-1} \overset{(\ref{eq:nhnmgdhGmtDd},i=n-2)}\Longrightarrow A_{n-1}\rblambda A_0.
\end{aligned}
\,\,\right\} 
\label{eq:tnttnksmLp}
\end{align}
By the first three lines of \eqref{eq:tnttnksmLp} and the transitivity of $\rblambda$, we have that $A_0\rblambda A_{n-1}$. But this contradicts the last line of  \eqref{eq:tnttnksmLp}, where $A_{n-1} \rblambda A_0$. The proof of Theorem \ref{thm:fncZ} is complete.
\end{proof}

\begin{figure}[ht]\centerline{ \includegraphics[scale=1.0]{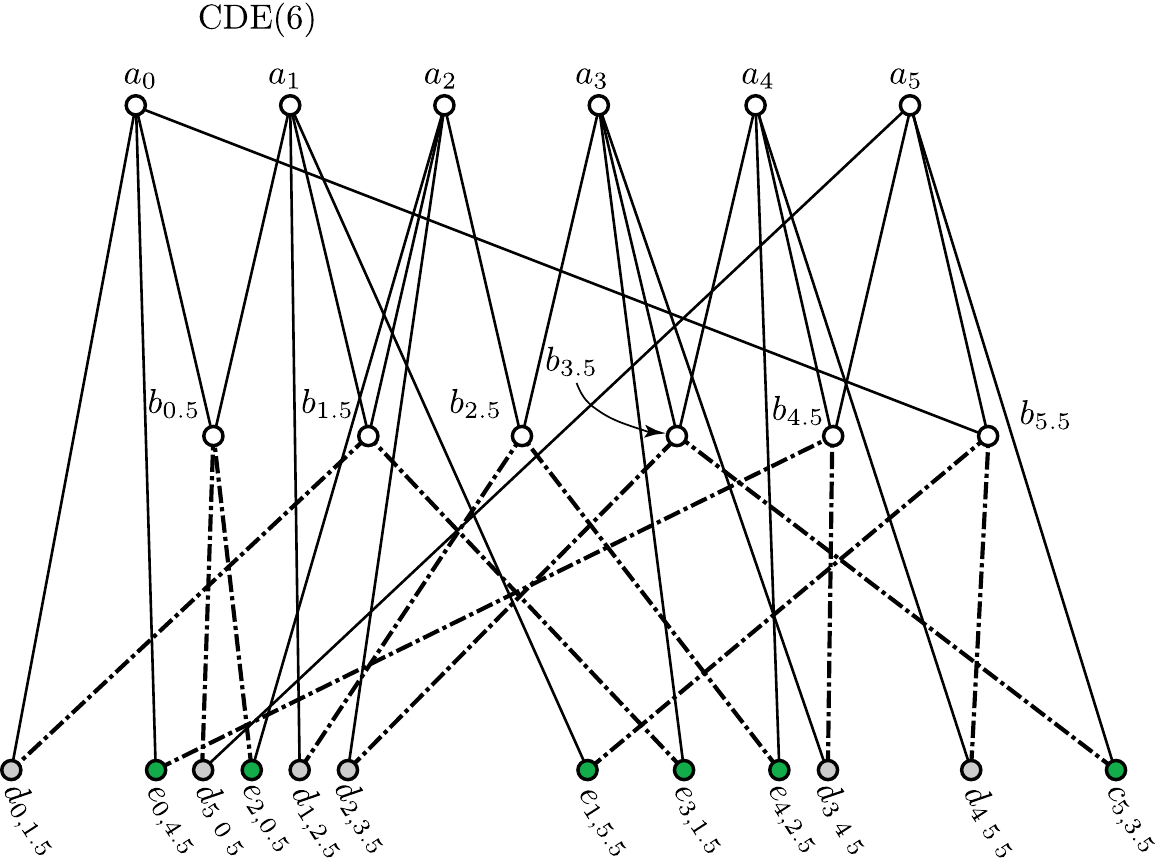}} \caption{$\CDE 6$}\label{fig10}
\end{figure}

\section{Another infinite family of new properties}\label{sect:CMP}
Following Cz\'edli and Schmidt \cite{CzGCSchT103}, \emph{patch lattices} are  slim rectangular lattices in which the corners are coatoms. 
These lattices have only two boundary lamps. Hence, if $L$ is a slim patch lattice, then $\Jir{\Con L}$ only has two maximal elements, whereby $\Jir{\Con L}$  trivially satisfies  $\CTF n$ for all $n\geq 2$. This means that Theorem \ref{thm:fncZ} says nothing on the congruence lattices of slim patch lattices. This observation motivates us to present another infinite family of properties; these properties are interesting even 
in the study of congruence lattices of slim \emph{patch} lattices.

For an integer $n\geq 3$, we define the poset \emph{Crown with Diamonds and Emeralds}\footnote{This terminology is explained by Figure \ref{fig11}, which is built on a picture from www.clker.com.} of order $n$, denoted by $\CDE n$, as follows; note that $\CDE 6$ is drawn in Figure \ref{fig10}.
First, let 
\[\ZK n:=\set{0,0.5,1,1.5,2,2.5,\dots, n-1, n-0.5};
\]
it is an additive abelian group and $\ZZ n$ is one of its subgroups. That is, we perform the addition and subtraction in $\ZK n$ modulo $n$.
For example, in $\ZK 4$, we have that $2+2=0.5+3.5=0$ and 
$1-3.5=1.5$. 
The underlying set of our poset is
\begin{align*}
\CDE n:= \set{a_i: i\in\ZZ n} &\cup \set{b_{i+0.5}: i\in\ZZ n}\cr
&\cup  \set{d_{i,i+1.5}: i\in \ZZ n} \cup \set{e_{i,i-1.5}: i\in \ZZ n}
\end{align*}
while its edges (that is, prime intervals)  are
$b_{i-0.5}\prec a_i$, $b_{i+0.5}\prec a_i$,
$d_{i,i+1.5}\prec a_i$, $d_{i,i+1.5}\prec b_{i+1.5}$,
$e_{i,i-1.5}\prec a_i$, and $d_{i,i-1.5}\prec b_{i-1.5}$, for $i\in\ZZ n$; see Figures \ref{fig10} and \ref{fig11} for illustration.
The elements of the forms  $a_i$, $b_i$, $d_{i,i+1.5}$, and $e_{i,i-1.5}$ of  $\CDE n$ are called \emph{maximal elements}, \emph{atoms}, \emph{diamonds}, and \emph{emeralds}, respectively. Note that $|\CDE n|=4n$. 

For a poset $P$ and an embedding $\phi\colon (\CDE n;\leq)\to (P;\leq)$, we say that $\phi$ \emph{preserves the coatomic edges} if  whenever $x\prec y$ in $\CDE n$ and $y$ is a maximal element (that is, $y$ is of the form $a_i$), then 
$\phi(x)\prec \phi(y)$ in $P$.  For example, any cover-preserving embedding $\CDE n\to P$ 
 preserves the coatomic edges but not conversely. 
 \red{We say that an order-preserving function $\phi\colon (\CDE n;\leq)\to (P;\leq)$ is a \emph{de-embedding} if its restriction to $\CDE n\setminus \{$diamonds$\}$ and its restriction to $\CDE n\setminus \{$emeralds$\}$ are order embeddings. 
 }

\begin{definition}\label{def:cmpn}
For an integer $n\geq 3$, we say that a poset $P$ satisfies the
\emph{$\CDE n$-property} if there exists no \red{de-}embedding $\phi\colon \CDE n\to P$  preserving the coatomic edges.
\end{definition}

Note that for $n=6$, the condition on $\phi$ is visualized in Figure \ref{fig10} as follows: $\phi$ has to preserve the coverings denoted by thin solid edges but it need not preserve the coverings indicated by the somewhat thicker ``dash-dot-dash-dot''-drawn edges.  Let us emphasize that if $\phi\colon (\CDE n;\leq)\to (P;\leq)$ preserves the coatomic edges, then the $\phi$-images of the maximal elements of $\CDE n$ need not be maximal in $P$. By an SPS lattice we still mean a slim semimodular lattice (which is necessarily planar).

\begin{figure}[ht] \centerline{ \includegraphics[scale=1.0]{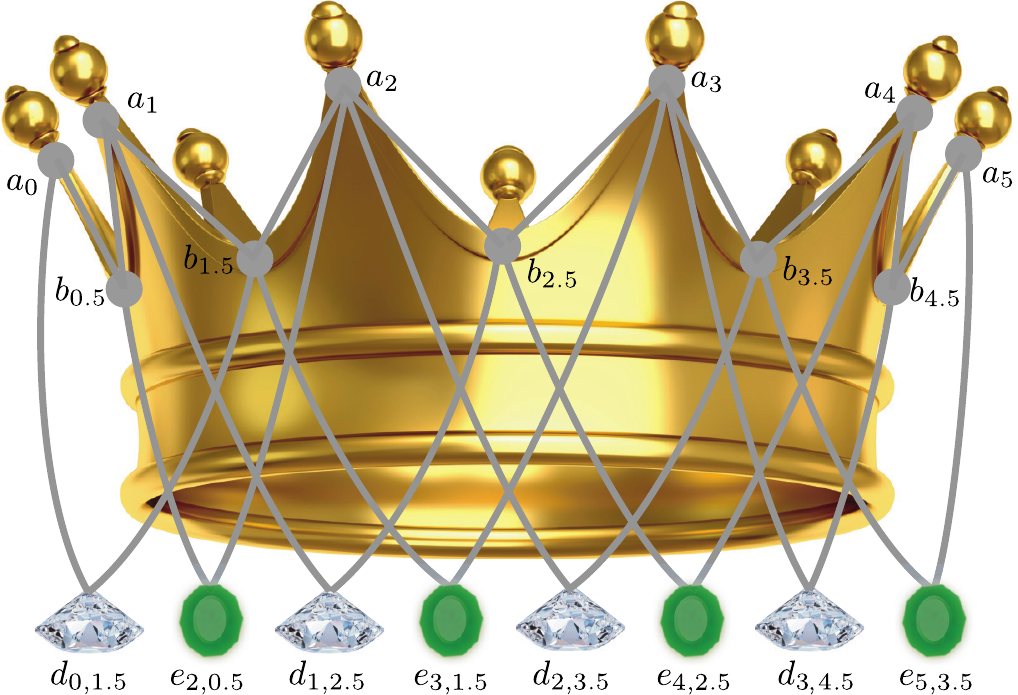}} \caption{Motivating the terminology by a subposet of $\CDE 9$}\label{fig11}
\end{figure}

\begin{theorem}\label{thm:CMP}
For every integer $n\geq 3$ and any SPS lattice $K$, 
$\Jir{\Con K}$ satisfies the $\CDE n$-property. 
\end{theorem}

\begin{proof}[Proof of Theorem \ref{thm:CMP}] To present a proof by way of contradiction, suppose that the theorem fails. By Lemma \ref{lemma:NTL}\eqref{lemma:NTLb}, $(\Jir L;\leq)\cong (\Lamp L;\leq)$ for a slim rectangular lattice $L$. Hence, there is \red{de-}embedding $\phi \colon \CDE n\to \Lamp L$ that preserves the coatomic edges. Again, for $x\in \CDE n$,  $X:=\phi(x)$. 
The disjunction of $A_{i_1}\rlambda A_{i_2}$ and $A_{i_2}\rlambda A_{i_1}$ is denoted by $A_{i_1}\gstparallel A_{i_2}$. (The subscript comes from ``geometrically left or right''.)
For $i\in\ZZ n$, $A_i$ and $A_{i+1}$ have a common lower cover, $B_{i+0.5}\in \Lamp L$. (Here and later,
the  arithmetics for indices is understood in $\ZK n$, that is, modulo $n$.) 
By  Lemma 4.3 of Cz\'edli and Gr\"atzer \cite{CzGGG3p3c}
(or by Lemma \ref{lemma:stngrrczg} and the last sentence of Lemma \ref{lemma:LsrTn}), 
\begin{equation}
\text{for every  $i\in\ZZ n$, $A_{i}\gstparallel A_{i+1}$.}
\label{eq:wmtbskSd}
\end{equation}

Next, we claim that 
\begin{equation}
\text{for every $i\in\ZZ n$, if  $A_i\rlambda  A_{i+1}$, then 
$A_{i+1}\rlambda  A_{i+2}$.}
\label{eq:nhfrnhlTr}
\end{equation}
To show this, suppose the contrary. 
Then, by \eqref{eq:wmtbskSd}, $A_{i}\rlambda A_{i+1}$ and $A_{i+2}\rlambda A_{i+1}$.
For the geometric relation between $A_i$ and $A_{i+2}$, the  (last sentence of) Lemma \ref{lemma:LsrTn} only allows four possibilities; we are going the exclude each of these four possibilities and then \eqref{eq:nhfrnhlTr} will follow by way of contradiction. 

First, let $A_{i+2}\rdelta A_i$. Then Lemma \ref{lemma:stngrrczg} applies since $A_{i+2}\rdelta A_i$, $D_{i,i+1.5}<B_{i+1.5}\prec A_{i+2}$, and we obtain that $D_{i,i+1.5}\nprec A_i$, a contradiction.

Second, let $A_i\rdelta A_{i+2}$. Then Lemma \ref{lemma:stngrrczg}
applies to $A_i\rdelta A_{i+2}$ and $E_{i+2,i+0.5}<B_{i+0.5}\prec A_i$, and we get a contradiction, $E_{i+2,i+0.5}\nprec A_{i+2}$.

Third, let $A_i\rlambda A_{i+2}$. Then $A_i\rlambda A_{i+2} \rlambda A_{i+1}$,  $B_{i+0.5}\prec A_i$, and $B_{i+0.5}\prec A_{i+1}$. Hence, Lemma \ref{lemma:xlhmQjlRn} implies that $B_{i+0.5} \rdelta A_{i+2}$. Now $B_{i+0.5} \rdelta A_{i+2}$, \  $E_{i+2,i+0.5}< B_{i+0.5}$, and Lemma \ref{lemma:stngrrczg} give that
$E_{i+2,i+0.5}\nprec A_{i+2}$, a contradiction again.

Fourth,  let $A_{i+2}\rlambda A_{i}$. Then
$A_{i+2}\rlambda A_{i} \rlambda A_{i+1}$, \ 
$B_{i+1.5}\prec A_{i+2}$, \  $B_{i+1.5}\prec A_{i+1}$,  and Lemma \ref{lemma:xlhmQjlRn} give that $B_{i+1.5}\rdelta A_i$. Hence, 
$B_{i+1.5}\rdelta A_i$, $D_{i,i+1.5}<B_{i+1.5}$, and Lemma \ref{lemma:stngrrczg} imply that $D_{i,i+1.5}\nprec A_i$, which is a contradiction. We have verified \eqref{eq:nhfrnhlTr}.

Finally, using \eqref{eq:wmtbskSd} and reflecting the diagram across a vertical axis if necessary, we can assume that $A_0\rlambda A_1$. Then, keeping in mind that $(n-1)+1=0$ in $\ZZ n$ and using \eqref{eq:nhfrnhlTr} repeatedly, we obtain that 
\begin{align*}
A_0\rlambda A_1 \rlambda A_2\rlambda A_3 \rlambda \dots \rlambda A_{n-1}\rlambda A_0.
\end{align*}
By the transitivity of $\rlambda$, see \eqref{eq:wknhTtPkrzrk}, it follows that $A_0\rlambda A_0$, which is a contradiction since $\rlambda$ is irreflexive by \eqref{eq:wknhTtPkrzrk}. This completes the proof of Theorem \ref{thm:CMP}.
\end{proof}

\section{Concluding remarks}\label{sect:conclR}
The $\CTF 2$-property is the same as the Two-pendant Four-crown Property, see Definition 4.1 and Theorem 4.3 in Cz\'edli \cite{CzGlamps}.

A poset $P$  has the Three-pendant Three-crown property, see 
 Cz\'edli and Gr\"atzer \cite{CzGGG3p3c}, if there is no cover-preserving embedding of \red{$\CDE 3\setminus\set{e_{0,1.5},e_{1,2.5},e_{2,0.5}}$} into $P$. Since the \red{de-}embedding need not be (fully) cover-preserving in Definition \ref{def:cmpn} 
 \red{and, in case of $\CDE 3$, it can collapse each diamond  with the emerald having the same subscript}, 
 the $\CDE 3$-property is stronger than the Three-pendant Three-crown property. Indeed, it is a trivial task to add \red{some new elements} to $\CDE 3$ to obtain a poset that satisfies 
the Three-pendant Three-crown property but fails to satisfy the $\CDE 3$-property. Therefore, the $n=3$ instance of Theorem \ref{thm:CMP} is stronger than the main result of  Cz\'edli and Gr\"atzer \cite{CzGGG3p3c}.

The smallest instance of the $\CTF n$-property and that of the $\CDE n$-property have been analysed. Hence, in the rest of this section, we assume that $n\geq 3$ for the $\CTF n$-property and $n\geq 4$ for the $\CDE n$-property even if this will not be mentioned explicitly.

For $k\geq 3$, let $D_k$ be the distributive lattice such that
$\Jir {D_k} \cong \CTF k$. Since crowns of different sizes cannot be embedded into each other, it is easy to see that
$D$ satisfies 
the seven previously known properties, the $\CTF n$-properties for all $n\neq k$, and the $\CDE n$-properties for all $n\geq 3$. However, $\CTF k$ fails in $D_k$. Similarly, if $k\geq 4$ and $D'_k$ is the distribute lattice defined by $\Jir{D'_k}\cong \CDE k$, 
then $D'_k$ satisfies the seven previously known properties, the $\CDE n$-properties for all $n\neq k$, and the $\CTF n$-properties for all $n\geq 3$.
Therefore, the $\CTF n$-properties, for $n\geq3$ and the $\CDE n$-properties, for $n\geq 4$ are new and we have an independent infinite set of properties of congruence lattice of SPS lattices. 

\red{The injectivity of $\phi$ means a plenty of conditions of the pattern ``if $x\neq y$ then $\phi(x)\neq\phi(y)$''. Observing that not all of these conditions are used in our proofs, it is possible to strengthen the new properties given in the paper and even some of the old properties. These details are elaborated in Cz\'edli \cite{CzGinprep}.}

\end{document}